\documentclass[reqno,oneside,11pt]{amsart}
\usepackage[T1]{fontenc}
\usepackage[utf8]{inputenc}
\usepackage{lmodern}
\usepackage[french,english]{babel}
\usepackage{geometry} 

\usepackage{amsmath,amssymb,amsfonts,amsthm,mathtools}
\usepackage{mathrsfs}
\usepackage{graphicx}
\usepackage{caption}
\usepackage{subcaption} 
\usepackage{booktabs}
\usepackage{array}
\usepackage{paralist}
\usepackage{fancyhdr}
\usepackage{emptypage}
\fancypagestyle{plain}{
	\fancyhf{}
	\fancyfoot[RO,RE]{\thepage}
	
	}

\usepackage{comment} 
\usepackage{diagbox}
\usepackage{xcolor}
\usepackage{epigraph}

\usepackage[noadjust]{cite}
\usepackage[
colorlinks=true,
urlcolor=purple,
linkcolor=purple!87!black,
citecolor=green!60!black,
pdfborder={0 0 0}
]{hyperref}
\usepackage{bookmark}
\usepackage{dsfont}

\newcommand{\cay}[2]{\mathrm{Cay}(#1,#2)}

\newcommand{\Z}{\mathbb{Z}}

\newcommand{\supp}[1]{\mathrm{supp}(#1)}

\newcommand{\std}{S_{\mathrm{std}}}
\newcommand{\TS}[3]{\mathrm{TS}\left(#1,#2,#3\right)}

\usepackage{tikz}
\usetikzlibrary{patterns,positioning,arrows,decorations.markings,calc,decorations.pathmorphing,decorations.pathreplacing}
\usetikzlibrary{fadings}
\usepackage{forest}
\usepackage{tikz-qtree}
\RequirePackage{pgffor} 
\makeatletter
\newcommand{\neutralize}[1]{\expandafter\let\csname c@#1\endcsname\count@}
\makeatother

\newcommand{\thistheoremname}{}

\newtheorem*{genericthm*}{\thistheoremname}
\newenvironment{namedthm*}[1]
{\renewcommand{\thistheoremname}{#1}%
	\begin{genericthm*}}
	{\end{genericthm*}}
\theoremstyle{plain}
\newtheorem{thm}{Theorem}[section] 
\newtheorem{proposition}[thm]{Proposition}
\newtheorem{lemma}[thm]{Lemma}
\newtheorem{corollary}[thm]{Corollary}
\theoremstyle{definition}
\newtheorem{remark}[thm]{Remark}
\newtheorem{definition}[thm]{Definition} 
\newtheorem{example}[thm]{Example} 
\newtheorem{question}[thm]{Question}

\allowdisplaybreaks
\address{Département de Mathématiques et Applications, \'{E}cole Normale Sup\'{e}rieure, PSL Research University, 45 rue d'Ulm, 75005, Paris, France.}
\email{eduardo.silva@ens.fr // edosilvamuller@gmail.com}
\title{Dead ends on wreath products and lamplighter groups}
\author{Eduardo Silva}
\date{\today}
\begin{document}

\newtheorem{thmy}{Theorem}
\renewcommand{\thethmy}{\Alph{thmy}} 
\newenvironment{thmx}{\stepcounter{thm}\begin{thmy}}{\end{thmy}}
\maketitle

\begin{abstract}
	For any finite group $A$ and any finitely generated group $B$, we prove that the corresponding lamplighter group $A\wr B$ admits a standard generating set with unbounded depth, and that if $B$ is abelian then the above is true for every standard generating set. This generalizes the case where $B=\mathbb{Z}$ together with its cyclic generator \cite{ClearyTaback05}. When $B=H*K$ is the free product of two finite groups $H$ and $K$, we characterize which standard generators of the associated lamplighter group have unbounded depth in terms of a geometrical constant related to the Cayley graphs of $H$ and $K$. In particular, we find differences with the one-dimensional case: the lamplighter group over the free product of two sufficiently large finite cyclic groups has uniformly bounded depth with respect to some standard generating set.
\end{abstract}

\section{Introduction}
\label{Introduction}

Let $G$ be a finitely generated group endowed with the word length $\|\cdot\|_S$ associated to a finite symmetric generating set $S$. By definition, any element $g\in G$ is the endpoint of a geodesic path from $e_G$ to $g$, and one can ask whether such paths can be extended beyond $g$, while remaining geodesic. Any element that fails to satisfy the above is called a \textit{dead end} of $G$ with respect to $S$.

The existence of dead end elements in a group is actually not so rare. For example, $\mathbb{Z}$ with the generating set $\{\pm 2, \pm 3\}$ has $1$ and $-1$ as dead ends of word length $2$. However, these are the only ones in this example, and in fact any abelian group always has finitely many dead ends \cite{lehnert2009some,vsunic2008frobenius}. For a virtually abelian group $G$ a slightly weaker condition holds: there exists a constant $M\ge 0$, which depends on the choice of $S$, such that any element $g\in G$ is at distance at most $M$ from a geodesic path of length $\|g\|_{S}+1$ that starts at $e_G$ \cite{warshall2010deep}. In such a case we say that $(G,S)$ has \textit{uniformly bounded depth}. In addition to virtually abelian groups, this property is satisfied for any choice of $S$ by hyperbolic groups \cite{Bogopolski97} and by groups with two or more ends \cite{lehnert2009some}. Intuitively, if $(G,S)$ has uniformly bounded depth, then geodesic paths starting at $e_G$ can be connected to longer ones at the cost of backtracking a constant number of steps.

If $(G,S)$ does not have uniformly bounded depth, we say that it has \textit{unbounded depth}. This means that one can find, for an arbitrary $n\ge 1$, elements $g\in G$ whose $n$-neighborhood is contained in the ball of radius $\|g\|_S$ centered at $e_G$. The existence of infinite groups with this property is not evident, and the first example was given by Cleary and Taback \cite{ClearyTaback05}, who showed that the lamplighter group $\Z/2\Z\wr \Z=\langle a,t\mid a^2, [a,t^iat^{-i}], i\in \mathbb{Z} \rangle$ has unbounded depth with respect to the generating set $\{a,t^{\pm 1}\}$. This result is a consequence of an explicit formula for the word length of elements in $\Z/2\Z\wr \Z$, as we explain now. In general, the word length in a wreath product $A\wr B$ with respect to standard generators can be expressed in terms of the word length in $A$ and the length of solutions to the Traveling Salesperson Problem (TSP) on the corresponding Cayley graph of $B$ (as shown by Parry in \cite{Parry1992}, and explained in Subsection \ref{subsection: wreath prods}). In the case of Cleary and Taback one has $B=\Z$ with generating set $\{t^{\pm 1}\}$, and hence the exceptionally simple solutions for the TSP on a line provide an explicit formula for the word length in $\Z/2\Z\wr \Z$. Similar arguments hold when replacing $\Z$ with a free group of finite rank together with a free generating set (see Section \ref{section: dead ends on finite and free groups}), but for other groups (or even free groups with other generating sets) this problem is in general computationally harder and one cannot hope for explicit solutions. Notably, the problem of finding the word length of a given element in $A\wr B$ is known to be NP-hard whenever $B$ is a finitely generated abelian group which contains $\Z^2$ \cite{KharlampovichMoghaddam2012}.
\subsection{Main results}

In this article, we study the depth properties of more general lamplighter groups $A\wr B$ with respect to standard generating sets, where $A$ is fixed to be a non-trivial group with unbounded depth for some generating set $S_A$ (in particular, $A$ can be any non-trivial finite group). Our first result states that standard generators with unbounded depth always exist.
\begin{thmx}[\;=\;Theorem \ref{thm: existence of genset with deep dead ends in any lamplighter}] \label{thm: Theorem A}
	Let $(A,S_A)$ have unbounded depth and $B$ be any finitely generated group. Then there exists a finite generating set $S_B$ of $B$ for which $(A\wr B,S_A\cup S_B)$ has unbounded depth.
\end{thmx}

Here it is essential that we consider \textit{standard} generating sets for the wreath product $A\wr B$. Indeed, Warshall has proved that the group $\Z/2\Z \wr \Z$ (and many other solvable groups) admit non-standard generating sets with uniformly bounded depth \cite{warshall2008strongly}.

We prove Theorem A in Section \ref{section:quasi-hamiltonian property} via the study of spanning paths of minimal length of finite subsets of the Cayley graph of $B$ with respect to $S_B$. We are interested in how this length varies when modifying the endpoint of said paths. Because of this, an important case for us is where balls of $B$ centered at the identity are close to being Hamiltonian-connected (i.e. there is a Hamiltonian path connecting any two vertices), up to repeating a constant number of elements (Definition \ref{def: qH property}). The existence of such a Cayley graph for an arbitrary group $B$ is proven in Lemma \ref{lem: existence of genset with quasiHamiltonian sequence}, and it is a consequence of the fact that the cube of any finite connected graph is Hamiltonian-connected (Lemma \ref{lem: the cube of a connected graph is Hamiltonian connected}).

Another family of graphs that is close to being Hamiltonian-connected are ``rectangular grids'' in $\Z^2$ (see Subsection \ref{subsection: abelian groups} and Lemma \ref{lem: Itai grid graphs} for precise definitions). By showing that there are bijective $1$-Lipschitz embeddings of these graphs onto the Cayley graph of any (infinite) abelian group, with the exception of $\Z$ with its cyclic generator, we obtain the following result: whenever $B$ is abelian, \textit{every} choice of $S_B$ gives rise to a standard generating set of $A\wr B$ with unbounded depth (Proposition \ref{prop: any lamplighter over an abelian group has unbounded depth wrt standard gensets}).

Next, we prove that the claim of Theorem A cannot in general hold for every standard generating set. That is, we find a finitely generated group $B$ together with a finite generating set $S_B$ for which the associated lamplighter group has uniformly bounded depth (Corollary \ref{cor: there exist a lamplighter group with uniformly bounded depth with respect to std gensets}). This example is explained in Subsection \ref{subsection: lamplighter with uniformly bounded depth}, and is obtained through the study of lamplighters over free products of finite groups. In order to formulate our results, we define for a finite group $G$ with a generating set $S_G$ the \textit{Hamiltonian difference} $\mathscr{H}(G,S_G)$, which measures how much shorter a minimal spanning cycle of $\cay{G}{S_G}$ is in comparison with minimal spanning paths from $e_G$ to a non-identity element (Definition \ref{def: new constant hamiltonicity finite groups}). The main result of Section $\ref{section: lamplighters over free products}$ is the following.

\begin{thmx}[\;=\;Theorem \ref{thm: general dead ends on free products}]\label{thm: Theorem B}
	Let $(A,S_A)$ have unbounded depth, and consider two finite groups $H$ and $K$ with generating sets $S_H$ and $S_K$, respectively. Then $(A\wr (H*K),S_A\cup S_H\cup S_K)$ has uniformly bounded depth if and only if $\mathscr{H}(H,S_H)+\mathscr{H}(K,S_K)\ge 1.$
\end{thmx}
In particular, whenever $\cay{H}{S_H}$ and $\cay{K}{S_K}$ are sufficiently long cycles, the lamplighter group $A\wr (H*K)$ has uniformly bounded depth (Corollary \ref{cor: free products of finite abelian groups characterization of lamplighter depth}). This seems to be the first examples of lamplighter groups with uniformly bounded depth with respect to standard generators, in contrast with the already mentioned example by Warshall of a non-standard generating set of $\Z/2\Z\wr\Z$ with uniformly bounded depth \cite{warshall2008strongly}.

\subsection{Background and ideas of the proofs}
In the remainder of the introduction we give more background about the study of dead ends in finitely generated groups, and explain the main ideas of the proofs.

The first definition of dead ends in the literature is commonly attributed to Bogopolski \cite{Bogopolski97}, and appears in his proof of the fact that two commensurable hyperbolic groups must be bi-Lipschitz equivalent (soon after, this result was proven to hold for non-amenable groups by Whyte \cite{Whyte99} and Nekrashevych \cite{Nekrashevych98}, without referencing dead ends in their proofs). Some other contexts where dead ends occur in the study of the geometry of Cayley graphs are mentioned in the following non-exhaustive list.
\begin{enumerate}
	\item If $(G,S)$ has unbounded depth, then the language of geodesic words with respect to $S$ cannot be regular \cite{warshall2010deep}.
	
	\item Dead ends appear as points of non-negative conjugation curvature, a notion of ``medium scale'' Ricci curvature for Cayley graphs introduced by Bar-Natan, Duchin and Kropholler \cite{BarNatanDuchinKro2020}, and often lead to finding elements of strictly positive conjugation curvature \cite{kropholler2020mediumscale}.
	
	\item Dead ends of arbitrarily large retreat depth are an obstruction to the connectedness of thickened spheres of Cayley graphs, as studied by Brieussel and Gournay \cite{BG18}.
	
	\item A zero asymptotic density of dead ends in the balls of the group is used as an assumption by Saito in \cite[Section 11.2, Assumption 2. \textbf{S}]{Saito10}. It was later remarked by Calegari and Fujiwara that this is quite restrictive, since there are hyperbolic groups with standard generating sets that have a positive density of dead ends (of uniformly bounded depth) \cite{CF15}.
\end{enumerate}

As we remarked above, the first known examples of groups with unbounded depth were provided by Cleary and Taback \cite{ClearyTaback05} using \textit{wreath products}, which we explain in more detail now. Given two groups $A$ and $B$, we define their \textit{wreath product} $A\wr B$ as the semidirect product $\bigoplus_{B} A\rtimes B$, where $B$ acts by translations on the group $\bigoplus_{B} A$ of finitely supported functions $f:B\to A$. We say that a generating set of $A\wr B$ is \textit{standard} if it is of the form $\std=S_A\cup S_B$, where $S_A$ and $S_B$ are generating sets of $A$ and $B$, respectively. 

Even though the Cayley graph of a wreath product $A\wr B$ with respect to a standard generating set is not completely understood (descriptions of Cayley graphs of lamplighter groups $\Z/n\Z \wr \Z$ as Diestel-Leader graphs are known for non-standard generators \cite{Woess2005}), the word metric can be described in terms of the ones of $A$ and $B$.  As we explain in Subsection \ref{subsection: wreath prods}, the word length of an element in $(A\wr B,\std)$ can be expressed in terms of the minimal length of paths in $\cay{B}{S_B}$ which start at $e_B$, visit some finite subset of elements $F\subseteq B$, and finish in some other one $x\in B$. These paths are solutions to the Traveling Salesperson Problem (TSP) in $\cay{B}{S_B}$, and we denote the length of a minimal such path by $\TS{e_B}{x}{F}$. More precisely, Equation \eqref{eq: express word length of wreath product} says that for $g=(f,x)\in A\wr B$,
$$
\|g\|_{\std}=\sum_{v\in \supp{f}}\|f(v)\|_{A} + \TS{e_B}{x}{\supp{f}}.
$$

A common way to interpret this formula is to think of $\cay{B}{S_B}$ as a street that has lamps at every vertex, where each lamp can be at a different state for each element of $A$, so that a group element $(f,b)\in A\wr B$ is given by a lamps configuration $f$ and a position $b\in B$. Then the generators of $S_B$ account for moving through the street, while the generators of $S_A$ change the state of the lamp at the current position. This is the origin of the name ``lamplighter group'', and we call $A$ the \textit{lamps group} and $B$ the \textit{base group}. When $\cay{B}{S_B}$ is a tree, it is possible to give a simple description of the solutions to the TSP inside the graph, and hence obtain an explicit formula for the word length in $A\wr B$  (Lemma \ref{lem: word length lamplighter over tree}). This is in particular used by Cleary and Taback, who studied lamplighter groups of the form $A\wr \Z$, where $A$ has unbounded depth with respect to some generating set $S_A$, and $S_{\Z}$ is the cyclic generating set of $\Z$. They proved that the standard generating set $\std=S_A\cup S_{\Z}$ has unbounded depth, and thus provided the first examples of groups with this property \cite{ClearyTaback05}. Their arguments strongly rely on the fact that the Cayley graph of the base group is a line, and generalizes to finitely generated free groups (Proposition \ref{prop: characterization of dead ends on lamplighters over free groups}). On the other hand, for other base groups, or even other generators of $\Z$, the TSP is known to be computationally hard and hence it is not possible to hope for an explicit description of the word length of a wreath product. 

Our results concern the existence of dead ends of arbitrarily large depth in wreath products $A\wr B$ over more general base groups $B$. Note that if the group $A$ has uniformly bounded depth with respect to $S_A$, then it is straightforward to see that $A\wr B$ also does with respect to $\std$. Because of this, we concentrate on the case where $(A,S_A)$ has unbounded depth and $B$ is any group with a finite generating set $S_B$. We use the name \textit{lamplighter group} to refer to any such wreath product. 

Our focus on dead ends leads us to study the value of $\TS{e_B}{v}{F}$ in finite connected subgraphs $F$ of $\cay{B}{S_B}$ that contain large balls centered at the identity $e_B$. In many cases, we observe that these solutions are actually Hamiltonian paths: they visit each vertex of $F$ exactly once, except possibly for one element when the path is a cycle (see Example \ref{example: qH Z^2}). Thanks to the Fuzz Lemma \ref{lem: Fuzz Lemma}, frequently used by Warshall for studying depth properties of groups \cite{warshall2008strongly,warshall2010deep,warshall2011group}, it is enough to estimate the word length up to an additive constant. This together with the above remark on Hamiltonian paths motivates the following definition.  A Cayley graph $\cay{B}{S_B}$, where $B$ is a group with a finite generating set $S_B$, is said to be \textit{quasi-Hamiltonian} if there exists a constant $M\ge 0$ such that for any $n\in \mathbb{N}$, there exists a finite subset $F\subseteq B$ which contains the ball $B_{S_B}(e_B,n)$ and for which $|\TS{e_B}{v}{F} - |F||\le M$, for any $v\in F$. This is, a path of minimal length which starts at $e_B$, visits all elements of $F$ and finishes at any $v\in F$, while visiting each element of $F$ at most once except for a bounded number of instances.

In Lemma \ref{lem: Hamiltonian sequence implies unbounded depth} we prove that if $\cay{B}{S_B}$ is a quasi-Hamiltonian presentation, then the corresponding lamplighter group $(A\wr B,\std)$ has unbounded depth. Then, using results about Hamiltonian-connectedness of finite graphs, we prove that any infinite group admits a quasi-Hamiltonian Cayley graph (Lemma \ref{lem: existence of genset with quasiHamiltonian sequence}), and hence that any lamplighter group $A\wr B$ admits a standard generating set with unbounded depth (Theorem \ref{thm: existence of genset with deep dead ends in any lamplighter}).

Having proved the existence of at least one quasi-Hamiltonian Cayley graph for any group, one may wonder if in general an arbitrary generating set will have this property. A clear restriction is that if $\cay{B}{S_B}$ is a tree then it cannot be quasi-Hamiltonian, since any path visiting all vertices in a ball is forced to repeat an unbounded amount of them (Example \ref{example: qH free group}). It turns out that this is the only constraint in the family of abelian groups: we prove that any Cayley graph of a finitely generated abelian group is quasi-Hamiltonian, with the exception of $\cay{\Z}{\{\pm 1\}}$ (Proposition \ref{prop: any cayley graph of a fg abelian group has a quasi Hamiltonian sequence}).  In order to show this result, we use the fact that grid graphs (finite induced subgraphs of $\Z^2$ with canonical generators, with vertex set $[0,n]\times[0,m]$) have spanning paths between any pair of vertices that repeat at most $2$ elements. This follows from a characterization of the existence of Hamiltonian paths in grid graphs due to Itai, Papadimitriou and Szwarcfiter \cite{ItaiPapadimitriouSzwarcfiter1982}. Then, an inductive argument shows that any Cayley graph of an abelian group, except for $\cay{\Z}{\{\pm 1\}}$, contains grid graphs as spanning subgraphs of sets containing arbitrarily large finite balls. By combining the above together with the original result of Cleary and Taback, we obtain as a corollary that any lamplighter group $A\wr B$ over an abelian group $B$ has unbounded depth, with respect to every standard generating set (Proposition \ref{prop: any lamplighter over an abelian group has unbounded depth wrt standard gensets}).

The study of Hamiltonian paths in Cayley graphs has a long history. Lovász Conjecture (for Cayley graphs) asks if any Cayley graph of a finite group has a Hamiltonian cycle and, although it has been verified for various families of groups, it remains far from being solved. We refer to Subsection \ref{subsection: Hamiltonian paths on finite Cayley graphs} for more details and to \cite{LanelPallageRatnayakeThevashaWelihinda2019} for a recent survey on the topic. On the other hand, for infinite groups there has been progress in the question of finding Hamiltonian paths or Hamiltonian circles in their Cayley graphs, that is, homeomorphic copies of the interval $[0,1]$ or the circle $S^1$, respectively, in the Freudenthal compactification of the graph \cite{MiraftabRuhmann2018}. We emphasize that our definition of quasi-Hamiltonian presentations concerns a slightly different question to the ones above since, despite our interest in infinite groups, we concentrate on paths covering finite (arbitrarily large) subgraphs of an infinite Cayley graph.

Even though Cayley graphs that are trees are not quasi-Hamiltonian, lamplighter groups over them still have unbounded depth with respect to standard generators (Proposition \ref{prop: characterization of dead ends on lamplighters over free groups}).  Hence it is natural to ask whether standard generating sets in lamplighter groups always have unbounded depth. We show that this is not the case, by constructing lamplighter groups which have uniformly bounded depth with respect to standard generators (Corollary \ref{cor: there exist a lamplighter group with uniformly bounded depth with respect to std gensets}). The Cayley graphs of the base groups of the above examples have cut vertices, that prevent them from being quasi-Hamiltonian, but at the same time contain sufficiently long cycles that allow an element to increase its word length by moving the position of the lamplighter. This construction seems to provide the first examples of lamplighter groups with uniformly bounded depth with respect to some standard generating set, in contrast with Warshall's results about non-standard generators with the same property \cite{warshall2008strongly}.

The above result is a particular case of our study of lamplighters over a free product of two finite groups $(H,S_H)$ and $(K,S_K)$. In this case, the depth properties of $A\wr (H*K)$ with respect to $\std=S_A\cup S_H\cup S_K$ are closely related to the solutions of the TSP in the finite graphs $\cay{H}{S_H}$ and $\cay{K}{S_K}$. More precisely, for a group $G$ with a generating set $S_G$ we define and study the \textit{Hamiltonian difference}
$$
\mathscr{H}(G,S_G)\coloneqq \max_{\substack{g\in G\backslash\{e_G\}}}\Big\{ \TS{e_G}{g}{G}\Big\}- \TS{e_G}{e_G}{G},
$$
where we recall that for any $g\in G$, $\TS{e_G}{g}{G}$ denotes the length of a path of minimal length in $\cay{G}{S_G}$ which starts at $e_G$, finishes at $g$, and visits all elements of $G$. When $\cay{G}{S_G}$ is Hamiltonian-connected we have $\mathscr{H}(G,S_G)=-1$, while on the other hand, $\mathscr{H}(G,S_G)$ can take any positive value if $\cay{G}{S_G}$ is chosen to be a sufficiently long cycle. The Hamiltonian difference measures how much shorter minimal spanning cycles are than minimal spanning paths from $e_G$ to a non-identity element inside $\cay{G}{S_G}$.

We prove that the value of $\mathscr{H}(H,S_H)+\mathscr{H}(K,S_K)$ completely characterizes the existence of dead ends of arbitrarily large depth in $(A\wr (H*K),\std)$ (Theorem \ref{thm: general dead ends on free products}), and we further detail the case of free products of finite abelian groups (Corollary \ref{cor: free products of finite abelian groups characterization of lamplighter depth}). Notably, a lamplighter over the free product of two sufficiently long cycles has uniformly bounded depth. Lamplighters over free products of finite groups also provide an interesting contrast to the strict structure of dead end elements of lamplighters over trees (Example \ref{example: free product unbounded depth arbitrarily far away position}).

The organization of the article is as follows. In Section \ref{section: preliminaries} we introduce the notation and basic tools we use. We also define wreath products and interpret the word length of an element in standard generating sets through solutions to the TSP in $\cay{B}{S_B}$ (Equation \ref{eq: express word length of wreath product}). Then, in Section \ref{section: dead ends on finite and free groups} we discuss lamplighter groups with $B$ finite and with $B$ a free group, which correspond to cases where the solutions to the TSP have a simple structure (with respect to our objectives of studying depth properties). We introduce the quasi-Hamiltonian property in Section \ref{section:quasi-hamiltonian property} and study its consequences on the depth properties of lamplighters. Notably, we prove Lemma \ref{lem: existence of genset with quasiHamiltonian sequence} about the existence of quasi-Hamiltonian Cayley graphs, and use it to prove Theorem \ref{thm: existence of genset with deep dead ends in any lamplighter}. The section finishes with Propositions \ref{prop: any cayley graph of a fg abelian group has a quasi Hamiltonian sequence} and \ref{prop: any lamplighter over an abelian group has unbounded depth wrt standard gensets}, regarding lamplighters over abelian groups. Finally, in Section \ref{section: lamplighters over free products} we study lamplighters over free products of finite groups, describe explicitly a lamplighter group with uniformly bounded depth for a standard generating set in Subsection \ref{subsection: lamplighter with uniformly bounded depth}, and then prove Theorem \ref{thm: general dead ends on free products}.

\section{Preliminaries}\label{section: preliminaries}
\subsection{Graphs}\label{subsection:graphs}

We start by recalling essential concepts of graph theory and by fixing our notation.

A graph $\Gamma$ is a pair $(V,E)$,
where $V=V(\Gamma)$ and $E=E(\Gamma)$ are the sets of vertices and edges of $\Gamma$, that is, $E$ consists of unordered pairs of vertices. We will work with graphs with finite as well as infinite sets of vertices and edges. For the purposes of this paper, these sets will always be countable and graphs will be locally finite, meaning that each vertex forms part of finitely many edges.

A path $P$ in $\Gamma$ is a sequence of (not necessarily distinct) vertices $P=v_1,v_2,\ldots, v_n\in V$ such that for all $1\le i <n$, there is an edge connecting $v_{i}$ to $v_{i+1}$, and we say that the length of $P$ is $n$. If moreover $v_1=v_{n}$, we say that $P$ is a cycle.

If $\Gamma$ is a finite graph, a \textit{spanning path} (resp. \textit{spanning cycle}) $P$ is one that visits each vertex of the $\Gamma$. If every vertex is visited a unique time (except for the final one in the case of a cycle), we call $P$ a \textit{Hamiltonian path} (resp.  \textit{Hamiltonian cycle}).

\begin{definition}\label{defn: Hamiltonian connected}
	A finite graph is said to be \textit{Hamiltonian} if it possesses a Hamiltonian cycle, and \textit{Hamiltonian-connected} if for any pair of distinct vertices, there is a Hamiltonian path connecting them.
\end{definition}

Any Hamiltonian-connected graph is of course Hamiltonian, but the opposite is not true: cycles of length $n\ge 4$ are Hamiltonian but no Hamiltonian-connected. This will be relevant for us in Section \ref{section: lamplighters over free products} when studying lamplighter groups over free products of cyclic groups.  

A general obstruction to Hamiltonian-connectedness is being bipartite, meaning that the vertex set $V(\Gamma)$ can be decomposed into two disjoint subsets $A$ and $B$, such that every edge in $E(\Gamma)$ connects a vertex in $A$ to a vertex in $B$. In that case a parity argument shows that, if the graph has at least $3$ vertices, there cannot be Hamiltonian paths between any two vertices. A bipartite graph with a partition $V(\Gamma)=A\cup B$ is called \textit{Hamiltonian-laceable} if there is a Hamiltonian path between any two vertices $u\in A$ and $v\in B$. This is our way of saying that $\Gamma$ has as many Hamiltonian paths as possible, given that it is bipartite.

When proving that a graph is not Hamiltonian-connected, common techniques usually rely on showing that a path visiting all vertices gets trapped at some vertex and must be forced to repeat other ones in order to finish where it is supposed to. This suggests that, if we allow the path to jump a finite bounded distance instead of only moving through adjacent vertices, we may be able to find paths which visit all vertices exactly once, with any starting and finishing point. In order to formalize that intuition, we make the following definition.

\begin{definition}\label{def: kth power of a graph} Given a finite connected graph $\Gamma$ and a positive integer $k\ge 1$, we define the $k$-th power graph $\Gamma^k$ of $\Gamma$ as the graph with $V(\Gamma^k)=V(\Gamma)$ and $$E(\Gamma^k)\coloneqq \{uv\mid u,v\in V(\Gamma)\text{ such that }d(u,v)\le k  \}.$$
	Commonly, $\Gamma^2$ is called the \textit{square} of $\Gamma$ and $\Gamma^3$ the \textit{cube} of $\Gamma$.
\end{definition}

The following result states that the cube of a finite connected graph is always Hamiltonian-connected. It was proved independently by Sekanina \cite{Sekanina1960} and Karaganis \cite{KaraganisCubeGraph1968}. This result can be proved by noting that it suffices to show it for a spanning tree of $\Gamma$, where an inductive argument can be applied.
\begin{lemma} \label{lem: the cube of a connected graph is Hamiltonian connected} Let $\Gamma$ be any finite connected graph. Then $\Gamma^3$ is Hamiltonian-connected.
\end{lemma}

Lemma \ref{lem: the cube of a connected graph is Hamiltonian connected} has been generalized to infinite graphs by Sekanina \cite{Sekanina1960}, who proved that the third power of any locally finite, $1$-ended graph has a spanning ray, and by Heinrich \cite{Heinrich1978}, who extended this fact to a class of non-locally finite graphs. With respect to Cayley graphs, Georgakopoulos used this result in \cite{GeorgakopoulosHamiltonian2009} to prove that any finitely generated group $G$ admits a finite generating set $S$ for which $\cay{G}{S}$ has a Hamiltonian circle.  We will use similar ideas in order to prove Theorem \ref{thm: existence of genset with deep dead ends in any lamplighter}.

The conclusion of Lemma \ref{lem: the cube of a connected graph is Hamiltonian connected} does not hold in general if we replace the cube of the graph by its square \cite[Figure 6.14]{GraphsandDigraphsbook2016}. However, Fleischner \cite{FleischnerTwoConnected1974,FleischnerSpanningSubgraphs1974} proved that if one adds the extra hypothesis that the graph is $2$-connected, then its square must have a Hamiltonian cycle. Moreover, Fleischner's result actually implies Hamiltonian-connectedness of the square of $2$-connected graphs, as proved by Chartrand, Hobbs, Jung, Kapoor and Nash-Williams \cite{ChartrandHobbsJungKapoorNashWilliams1974}.

We finish this subsection by giving the formal definition of the direct product of two graphs, which will be useful in Subsection \ref{subsection: abelian groups} when discussing Hamiltonian-connected properties of Cayley graphs of infinite abelian groups.
\begin{definition}\label{def: product graph}  Let $\Gamma_1,\Gamma_2$ be two graphs. We define their product $\Gamma=\Gamma_1\times \Gamma_2$ as the graph whose vertex set is $V(\Gamma)=V(\Gamma_1)\times V(\Gamma_2)$ and where two vertices $(u_1,u_2), (v_1,v_2)\in V(\Gamma)$ are connected by an edge if and only if $u_1=v_1$ and $u_2$ is connected by an edge in $E(\Gamma_2)$ to $v_2$, or if $u_2=v_2$ and $u_1$ is connected by an edge in $E(\Gamma_1)$ to $v_1$.
\end{definition}

\subsection{Groups}\label{subsection:groups}
Whenever we talk about groups, we assume that they are finitely generated. We denote by $(G,S)$ a group together with a finite (symmetric) generating set $S$. We use the notation $e_G$ for the identity element of the group $G$, or simply $e$ if there is no risk of confusion.

The (right, undirected, unlabeled) Cayley graph $\cay{G}{S}$ of $G$ with respect to the generating set $S$ is the graph whose vertices are the elements of $G$, and where two elements $g,g'$ are connected through an edge if and only if there exists $s\in S\cup S^{-1}$ with $g=g's$. In this context, a natural metric arises in $G$. Indeed, define for $g,h\in G$,
\begin{align*}
d_S(g,h)\coloneqq \min\left\{n\ge 0\mid g^{-1}h=s_1\cdots s_n,\text{ for some }s_1,\ldots,s_n\in S\cup S^{-1}\right\}.
\end{align*}
The distance $d_S$ is called the word metric on $G$ associated to the generating set $S$ and it corresponds to the length of a minimal path in $\cay{G}{S}$ connecting $g$ to $h$. 
Similarly, we define the word length associated to $S$ as
$$
\|g\|_{S}=d_S(e_G,g), \text{  for  }g\in G.
$$

Given $g\in G$ and $n\ge 0$, we define the ball of radius $n$ centered at $g$ as
$$
B(g,n)\coloneqq \{ h\in G\mid d_S(h,g)\le n \}.
$$ 
When there is risk of confusion, we use the notation $B_S(g,n)$ or $B_{(G,S)}(g,n)$ to emphasize the generating set or the group used to define the ball.

In Section \ref{section:quasi-hamiltonian property} we will be interested in finding, for each $n\ge 1$, spanning paths of finite subgraphs of $\cay{G}{S}$ containing $B(e_G,n)$, whose length is close to the length of a hypothetical Hamiltonian path, up to a uniform additive error. 

\subsection{Hamiltonian paths on finite Cayley graphs}\label{subsection: Hamiltonian paths on finite Cayley graphs}
The problem of finding Hamiltonian cycles on Cayley graphs of finite groups was first proposed by Elvira Rapaport Strasser \cite{Rapaport1959}, and then by Lovász in 1969 \cite[Appendix IV. Problem 20]{BondyMurty1979}. Lovász conjectured that every finite connected vertex-transitive graph has a Hamiltonian cycle, except for five known counterexamples: the complete graph on $2$ vertices, the Petersen graph, the Coxeter graph, and the graphs obtained by replacing in one of the last two graphs each vertex by a triangle. None of these counterexamples are Cayley graphs of groups, and hence the version Lovász conjecture for Cayley graphs of finite groups asks if any such graph with at least $3$ elements has a Hamiltonian cycle. Up until now this conjecture remains open, although it has been verified for various families of groups. Surveys on this topic can be found in \cite{LanelPallageRatnayakeThevashaWelihinda2019,WitteGallian1984,CurranGallian1996}. Notably, it is a well established fact that any Cayley graph of a finite abelian group with at least $3$ elements has a Hamiltonian cycle \cite{Marusic1983}.

With respect to properties such as Hamiltonian-connectedness or Hamiltonian-laceability, it is quick to find Cayley graphs which have none of these two properties. For example, any cycle of even length $\ell$, with $\ell >5$, is neither Hamiltonian-connected nor Hamiltonian-laceable, and hence such examples are found even within the family of finite cyclic groups. However, in 1981 Chen and Quimpo proved that among finite abelian groups these are the only counterexamples one can find.

\begin{proposition}[\cite{ChenQuimpo1981}] \label{prop: cayley graphs of abelian finite group are either cycles, Hamiltonian connected or Hamiltonian laceable } Let $\Gamma$ be a Cayley graph of a finite abelian group. Then if $\Gamma$ is not a cycle, either
	\begin{enumerate}
		\item $\Gamma$ is non-bipartite and Hamiltonian-connected, or
		\item $\Gamma$ is bipartite and Hamiltonian-laceable.
	\end{enumerate}
\end{proposition}

It is natural to ask if the conclusion of Proposition \ref{prop: cayley graphs of abelian finite group are either cycles, Hamiltonian connected or Hamiltonian laceable } holds in general, that is, if any Cayley graph of a finite group $G$ of degree at least $3$ is either Hamiltonian-connected, or it is bipartite and Hamiltonian-laceable \cite[Questions 4.1-4.3]{DupuisWagon2015}. As we saw above, this holds if $G$ is abelian, and it has also been proved using computational methods for groups of order $|G|<48$ \cite{WitteWilk2020}. Some other particular families of groups have been shown to satisfy this property \cite{Araki2006,Alspach2015,AlspachChenMcAvaney1996,AlspachQin2001}, but the general case is far from being solved.

\subsection{Wreath products and lamplighter groups}\label{subsection: wreath prods}\label{subsection:wreath prods}

For $A,B$ groups, we define their \textit{wreath product} $A\wr B$ as the semidirect product $\bigoplus_{B} A \rtimes B$, where $\bigoplus_{B} A $ is the group of finitely supported functions $f:B\to A$ endowed  with the operation $\oplus$ of componentwise multiplication. We denote by $\supp{f}$ the finite subset of $B$ to which $f$ assigns non-trivial values. Here, the group $B$ acts on the direct sum $\bigoplus_{B} A $ from the left by translations. That is, for $f:B\to A$, and any $b\in B$ we have
$$
(b\cdot f)(x)=f(b^{-1}x), \ x\in B.
$$

Elements of $A\wr B$ can be expressed as tuples $(f,b)$, where $f:B\to A$ is a finitely supported function and $b\in B$, and the product between two such elements elements $(f,b)$, $(f',b')\in A\wr B$ is given by
$$
(f,b)\cdot(f',b')= (f\oplus (b\cdot f'),bb').
$$
There is a natural embedding of $B$ into $A\wr B$ via the mapping
\begin{align*}
B&\to A\wr B\\
b&\mapsto (\mathds{1},b),
\end{align*}
where $\mathds{1}(x)=e_A$ for any $x\in B$. Similarly, we can embed $A$ into $A\wr B$ via the mapping
\begin{align*}
B&\to A\wr B\\
a&\mapsto (\delta^{a}_{e_B},e_B),
\end{align*}
where $\delta^{a}_{e_B}(e_B)=a$ and $\delta^{a}_{e_B}(x)=e_A$ for any $x\neq e_B$.

In particular, if we consider finite symmetric generating sets $S_A$ and $S_B$ of $A$ and $B$, respectively, their copies inside $A\wr B$ through the above embeddings generate the entire group $A\wr B$. We call $S_{\mathrm{std}}\coloneqq S_A\cup S_B$ the \textit{standard generating set} for $A\wr B$ associated to the generators $S_A$ and $S_B$. 

In order to understand how the word length of an element with respect to $\std$ looks like, consider $g=(f,x)\in A\wr B$ and write it as a product of generators of $\std$,
$$
g=a_0b_1a_1b_2a_2\cdots b_ma_m,
$$
with $m\ge 0$, $a_0,a_m\in S_A\cup\{ e_A \}$, $a_1,\ldots,a_{m-1}\in S_A$, and $b_1,\ldots,b_m\in S_B$. In particular, it holds that $$f=a_0(b_1\cdot a_1)(b_1b_2\cdot a_2)\cdots (b_1b_2\cdots b_m \cdot a_m),$$ 
so that $\supp{f}\subseteq \{e_B,b_1,b_1b_2,\ldots,b_1b_2\cdots b_m\}$, and $x=b_1b_2\cdots b_m$. This factorization of $g$ into generators of $\std$ can be interpreted as a path in the Cayley graph $\cay{B}{S_B}$ which begins at $e_B$, visits all vertices in $\supp{f}$ while generating the appropriate group element for $f$ in each one, and ends at $b_1b_2\cdots b_m$. This is the reason behind the name ``lamplighter group'': one can think of the Cayley graph $\cay{B}{S_B}$ as a street with lamps at every vertex, each of which can be in a different state given by an element of $A$. Then a word in $\std$ evaluating to an element $g=(f,x)\in A\wr B$ corresponds to a path from the origin $e_B$ to $x$, which passes through all vertices of $\supp{f}$ and at each one of them uses the generators of $S_A$ in order to reach the value of $f$ there. We refer to $x$ as the \textit{position of the lamplighter}, and to $f$ as the \textit{lamps configuration}.

The above discussion shows that the word length of an element $g=(f,x)\in A\wr B$ with respect to $\std$ can be expressed as
\begin{equation}\label{eq: express word length of wreath product}
\|g\|_{\std}=   \sum_{y\in \supp{f}}\|f(y)\|_{S_A}+\TS{e_B}{x}{f},
\end{equation}
where $\TS{b}{b'}{f}$ corresponds to the length of a path of minimal length in $\cay{B}{S_B}$ which starts at $b$, finishes at $b'$ and visits all vertices of $\supp{f}$. The notation $\mathrm{TS}$ stands for the interpretation of said walk as a solution to the Traveling Salesperson Problem (TSP). We also write $\TS{b}{b'}{F}$ to denote a walk of minimal length in $\cay{B}{S_B}$ starting at $b$, finishing at $b'$ and visiting all vertices from a finite subset $F\subseteq B$.

Equation \eqref{eq: express word length of wreath product} has been widely used to study wreath products. We already mentioned in the introduction that Cleary and Taback used it to study depth properties of lamplighter groups $A\wr \Z$ \cite{ClearyTaback05}, but it has also been used to study other metric properties of wreath products, as for example by Parry in order to study rationality and algebraicity of growth series \cite{Parry1992}, by Davis and Olshanskii to study distortion of subgroups of some wreath products \cite{DavisOlshanskii2011}, among many others.

\subsection{Dead ends on groups}\label{subsection: dead ends on groups}
Fix $(G,S)$ a group together with a finite generating set. We say that an element $g\in G$ is a \textit{dead end} if for any $s\in S\cup S^{-1}$, $\|gs\|_S\le \|g\|_S$, and we define its \textit{depth} with respect to $S$ as the maximal number $n\ge 1$ such that for any choice of generators $s_1,\ldots,s_k\in S\cup S^{-1}$, $k\le n$, we have $\|gs_1\cdots s_k\|_S\le\|g\|_S$. That is, the depth of a dead end is the maximal number $n$ such that multiplying by at most $n$ generators does not increase the word length of $g$. Another way of saying this is that $g$ maximizes the function $\|\cdot\|_S$ on its $n$-neighborhood.

Although their origin might be older, the first definition of dead ends is commonly attributed to Bogopolski in 1997 \cite{Bogopolski97}, who used it while proving that two commensurable hyperbolic groups must be bi-Lipschitz equivalent. Soon after, this property was shown to hold for any non-amenable groups by Whyte \cite{Whyte99} and Nekrashevych \cite{Nekrashevych98} without using the notion of dead ends. 

Ideas related to dead ends had already appeared before 1997 in the literature, as for example by Champetier in \cite[Lemme 4.19]{Champetier95} where it is proven that group presentations $G=\langle S\mid R\rangle$ satisfying the $C'(1/6)$ small cancellation condition (see \cite[Chapter V.2]{LyndonSchupp2001} for a definition) satisfy the following property: for any $g\in G$, the set $\{ s\in S\cup S^{-1}\mid \|gs\|_S\le \|g\|  \}$ has at most two elements. This is also discussed by de la Harpe in \cite[Chapter IV.A. 13,14]{delaHarpe2000}. In the latter, dead ends are introduced as an obstruction for the \textit{extension property for geodesic segments} of a graph. With respect to more recent literature, dead ends are discussed in the chapters of some books on Geometric Group Theory, as in \cite[Subsections 1.8.5, 2.6.4 \& 4.7.2]{sampling2018} and in \cite[Chapters 12, 15 \& 16]{officehours17}.

The depth of a dead end $g\in G$ with respect to a finite generating set $S$ can be interpreted as the distance in $\cay{G}{S}$ between $g$ and the complement of the ball $B_S(e_G,\|g\|_{S})$, minus $1$. Note that if $G$ is an infinite group, then the depth of any element $g$ has as an upper bound $2\|g\|_S$ (since any infinite finitely generated group contains an infinite geodesic ray). On the other hand, finite groups always have elements of infinite depth: those that maximize the word metric $\|\cdot \|_{S}$. However, although the depth of each element must be finite, it can be possible that $G$ contains dead ends of arbitrarily large depth.

\begin{definition}\label{def: unbounded depth} Let $G$ be a finitely generated group and $S$ a finite generating set. We say that $G$ has \textit{unbounded depth} if for any $n\in \mathbb{N}$ there exists a dead end $g\in G$ of depth at least $n$. Otherwise, we say that $G$ has \textit{uniformly bounded depth with respect to $S$.} 
\end{definition}

We again emphasize the dependence of these definitions on the choice of the generating set $S$. It may happen that a group $G$ has unbounded depth with respect to a generating set $S$, and uniformly bounded depth with respect to another generating set $S'$. In fact, \v{S}unić proved in \cite{vsunic2008frobenius} that any group admits a generating set with dead ends. In the same paper, he proved that $\Z$ always has finitely many dead ends, a fact that would later be proved to hold for any finitely generated abelian group by Lehnert \cite{lehnert2009some}. In particular, there are no generating sets with unbounded depth among such groups. Other families of groups which have uniformly bounded depth with respect to any generating set are hyperbolic groups \cite{Bogopolski97}, and more generally any group which has a regular language of geodesics for any generating set \cite{warshall2010deep}, and groups with more than one end \cite{lehnert2009some}. On the other hand, examples of group with unbounded depth are notably the lamplighter group over the line $\Z/2\Z\wr \Z$ \cite{CET2006,ClearyTaback05,ClearyTabackmetricproperties05}, and Houghton's group $H_2$ \cite{lehnert2009some}. In general, having unbounded depth is not a group invariant, since there exist groups (even finitely presented ones) which have unbounded depth for one generating set and uniformly bounded depth for another one \cite{riley2006unbounded,cleary2006finitely}. A remarkable exception is the discrete Heisenberg group, which has unbounded depth with respect to any finite generating set, as shown by Warshall \cite{warshall2010deep,warshall2011group}.

In order to prove the existence of dead ends in a Cayley graph, it is not necessary to find them explicitly: it suffices to show the existence of elements which increase their word length by a bounded amount when multiplying by a large number of generators. The following lemma formalizes this, and it has been widely used to show the existence of dead ends of arbitrary depth by Warshall \cite{warshall2008strongly,warshall2010deep,warshall2011group}.

\begin{lemma}[Fuzz Lemma]\label{lem: Fuzz Lemma}
	Let $X$ be a metric space, and $f:X\to \mathbb{Z}$ a function. Suppose there exists $M> 0$ such that for some $x\in X$ and $r\in \mathbb{N}^{+}$ we have
	\begin{equation*}\label{eq: condition for Fuzz lemma}
	f(x')\le f(x)+M, \ \text{ for all }x'\in B(x,r).
	\end{equation*}
	Then there exists some $x_0\in X$ such that $f$ attains a maximum on $B(x_0,r/M)$ at $x_0$.
\end{lemma}

To finish this section we introduce a slightly different notion of depth for dead ends, concerned with how much actual backtracking is needed in order to reach elements of bigger word length.
\begin{definition}\label{def: retreat depth}
	Given a dead end $g\in G$ with respect to a generating set $S$, we say that $g$ has \textit{retreat depth} (or \textit{strong depth}) $k$ if $k$ is the minimal number such that there exists a geodesic from $g$ to an element of $B_S(e_G,\|g\|_{S}+1)$ which does not pass through $B_S(e_G,\|g\|_{S}-k-1)$.
\end{definition} That is, the retreat depth of $g$ measures how many steps back in $\cay{G}{S}$ from $g$ one needs to take in order to eventually reach a bigger sphere. The retreat depth of an element is bounded above by its depth, but it may be the case that elements of arbitrarily large depth have uniformly bounded retreat depth. Indeed, this is the case for the discrete Heisenberg group: Warshall proved that it has unbounded depth and at the same time uniformly bounded retreat depth, for any generating set \cite{warshall2011group}. On the other hand, the lamplighter group $\Z/2\Z\wr\Z$ and Houghton's group $\mathcal{H}_2$ have unbounded retreat depth for standard generating sets \cite{lehnert2009some}.

Throughout this article, we work with wreath products $A\wr B$ under the assumption that $A$ has an associated finite generating set $S_A$ with unbounded depth, and call it the \textit{lamps group} of the wreath product $A\wr B$. This generalizes the case of finite lamps groups, since finite groups always have elements of infinite depth. We also say that $B$ is the \textit{base group} of $A\wr B$ or that $A\wr B$ is a lamplighter group over $B$ with lamp groups $A$. 
\section{Lamplighter groups over finite groups and over free groups}\label{section: dead ends on finite and free groups}
This section concerns two particular cases of base groups for which there is a detailed description of the solutions of the TSP, with respect to our purposes of studying depth on lamplighter groups. We show that when the base group is finite, the depth properties are dominated by the lamps group, while on the other hand when the base group is a free group with free generating set (so that the corresponding Cayley graph is a tree) we are able to give a precise characterization of dead end elements using the same arguments of \cite{ClearyTaback05}. 

In particular, it follows from Proposition \ref{prop: characterization of dead ends on lamplighters over free groups} that dead end elements of lamplighters over trees must necessarily have the position of the lamplighter at the identity element. This restriction does not hold in general for other base groups, and through our study of lamplighters over free products of finite groups in Section \ref{section: lamplighters over free products}, we find a lamplighter group with a standard generating set whose dead end elements can have the position of the lamplighter at an arbitrarily large distance from the identity element (Example \ref{example: free product unbounded depth arbitrarily far away position}).
\subsection{Lamplighters over finite groups}
We start by proving that when $B$ is finite, the lamplighter group $(A\wr B,\std)$ has the same depth properties as $(A,S_A)$. 
\begin{proposition}\label{prop: lamplighter over finite group deep dead ends if and only if lamps have deep dead ends} Consider $(A,S_A)$ any finitely generated group and $(B,S_B)$ a finite group. Then $(A\wr B,\std)$ has unbounded depth if and only if $(A,S_A)$ does. 
\end{proposition}
\begin{proof}
	
	Suppose first that $(A,S_A)$ has uniformly bounded depth, so that for some $k\ge 1$ and any $a\in A$, there exist $\alpha_1,\ldots,\alpha_k\in S_A\cup S_A^{-1}$ so that $\|a\alpha_1\cdots\alpha_k\|_{S_A}\ge \|a\|_{S_A}+1$.
	
	For any element $g=(f,x)\in A\wr B$, Equation \eqref{eq: express word length of wreath product} tells us that the word length of $g$ is
	$$
	\|g\|_{\std}=\|f\|_{S_A}+\TS{e_B}{x}{f}.
	$$
	For $a=f(x)$, find $\alpha_1,\ldots,\alpha_k\in S_A\cup S_A^{-1}$ so that$\|a\alpha_1\cdots\alpha_k\|_{S_A}\ge \|a\|_{S_A}+1$. Then
	
	$$
	g\alpha_1\cdots\alpha_k=(f\cdot x\alpha_1\cdots\alpha_kx^{-1},x),
	$$
	and hence
	\begin{align*}
	\|g\alpha_1\cdots\alpha_k\|_{\std}&=\|f\cdot x\alpha_1\cdots\alpha_kx^{-1}\|_{S_A}+\TS{e_B}{x}{f\cdot x\alpha_1\cdots\alpha_kx^{-1}}\\
	&\ge\|f\cdot x\alpha_1\cdots\alpha_kx^{-1}\|_{S_A}+\TS{e_B}{x}{f}\\
	&\ge \|f\|_{S_A}+1+\TS{e_B}{x}{f} \\
	&=\|g\|_{\std}+1,	
	\end{align*}
	where the second inequality comes from the fact that the support of $f$ is contained in that of $f$ with the state of the lamp at position $x$ modified. This shows that $g$ has depth at most $k$.
	
	Now suppose that $(A,S_A)$ has unbounded depth, and choose $a\in (A,S_A)$ of depth  at least $n$. Choose $x\in B$ that maximizes the value of $\TS{e_B}{x}{B}$, which exists since we assume $B$ to be finite. Then it is straightforward to prove that the element $(f,x)$, where $f(y)=a$ for all $y\in B$, is a dead end of depth at least $n$, by using Equation \eqref{eq: express word length of wreath product}. Indeed, thanks to our choice of $x$, the word length of $(f,x)$ can only be increased through the term associated to the lamps configuration, for which at least $n$ generators of $S_A$ are needed.
	
	Note that here it is essential that $B$ is finite, so that $f$ defined as above does indeed give a finitely supported function over $B$.
\end{proof}
For the rest of the article, we concentrate on lamplighter groups over infinite base groups $B$.
\subsection{Lamplighters over free groups}
Computing the word length of an arbitrary element using Equation \eqref{eq: express word length of wreath product} is not easy in general since it involves solving the TSP, known to be a computationally hard problem. However, an exceptionally simple formula can be given for $F(S)$ a free group with a free finite generating set $S$. In order to do so, we need to introduce some notation. Given $u,v\in F(S)$, denote by $[u,v]$ the set of edges of the unique shortest path in $\cay{F(S)}{S}$ joining $u$ to $v$, so that $d_{S}(u,v)=|[u,v]|$. Similarly, for $H\subseteq F(S)$ finite, denote $[u,H]=\bigcup_{h\in H}[u,h]$.
\begin{lemma}[{\cite[Theorem 3.1]{BaudierMotakisSchumprechtZsak2021}},\cite{ClearyTaback05}]\label{lem: word length lamplighter over tree} Let $S$ be a finite set and $F(S)$ the free group over $S$. Then for any $u,v\in F(S)$ and any finite subset $H\subseteq F(S)$ we have
	$$
	\TS{u}{v}{H}=2| [u,H]\backslash [u,v] |+|[u,v]|.
	$$
	
	In particular, for any element $g=(f,x)\in A\wr F(S)$, Equation \eqref{eq: express word length of wreath product} takes the form
	\begin{equation}\label{eq: word length wreath prod over tree}
	\|g\|_{\std}=\sum_{y\in\supp{f}}\|f(y)\|_{S_A}+2| [e_B,\supp{f}]\backslash [e_B,x] |+ \|x\|_S.
	\end{equation}
\end{lemma}

By using Equation \eqref{eq: word length wreath prod over tree} we can generalize the results about dead ends of lamplighter groups over the line from \cite{ClearyTaback05} and actually give a characterization of such elements.

\begin{proposition}\label{prop: characterization of dead ends on lamplighters over free groups} Consider the free group $F(S)$ over a finite set $S$, and a finitely generated group $(A,S_A)$. Then $g=(f,x)\in (A\wr F(S),\std)$ is a dead end if and only if
	\begin{enumerate}
		\item $f(x)\in (A,S_A)$ is a dead end,
		\item $[e_{F(S)},\supp{f}]$ contains all edges $[x,xs]$, for $s\in S^{\pm 1}$, and
		\item $x=e_{F(S)}$.
	\end{enumerate}
	Moreover, if $(A,S_A)$ has unbounded (retreat) depth, then $(A\wr F(S),\std)$ also does.
\end{proposition}

\begin{proof}
	Suppose first that the three conditions hold. As $x=e_{F(S)}$, Equation \eqref{eq: word length wreath prod over tree} says that
	$$
	\|g\|_{\std}=\sum_{y\in \supp{f}}\|f(y)\|_{S_A}+2|[e_{F(S)},\supp{f}]|.
	$$
	As $f(x)$ is a dead end of $(A,S_A)$, multiplying by a generator in $S_A$ does not increase word length. On the other hand, multiplying by a generator $s\in S$ gives
	\begin{align*}
	\|gs\|_{\std}&=\sum_{y\in \supp{f}}\|f(y)\|_{S_A}+2|[e_{F(S)},\supp{f}]\backslash[e_B,s]|+|[e_B,s]|\\
	&=\sum_{y\in \supp{f}}\|f(y)\|_{S_A}+2|[e_{F(S)},\supp{f}]\backslash[e_B,s]|+1\\
	&=\sum_{y\in \supp{f}}\|f(y)\|_{S_A}+2\left(|[e_{F(S)},\supp{f}]|-1\right)+1\\
	&=\|g\|_{\std}-1,
	\end{align*}
	where we used Condition (2) in the penultimate equality. This proves that $g$ is a dead end.

	Now let us suppose that $g=(f,x)\in (A\wr F(S),\std)$ is a dead end. Clearly Condition $(1)$ must hold, since otherwise multiplying by a generator in $S_A$ would increase word length. Now consider Condition (2). If there exists $s\in S^{\pm 1}$ such that the edge $[x,xs]$ is not contained in $[e_{F(S)},\supp{f}]$, then 
	\begin{align*}
	\|gs\|_{\std}&=\sum_{y\in \supp{f}}\|f(y)\|_{S_A}+2|[e_{F(S)},\supp{f}]\backslash[e_B,xs]|+|[e_B,xs]|\\
	&=\sum_{y\in \supp{f}}\|f(y)\|_{S_A}+2|[e_{F(S)},\supp{f}]\backslash[e_B,x]|+|[e_B,x]|+1\\
	&=\|g\|_{\std}+1,
	\end{align*}
	which contradicts the fact that $g$ is a dead end. Hence Condition $(2)$ holds.
	
	Similarly, suppose that $x\neq e_{F(S)}$ and choose $s\in S^{\pm 1}$ such that $\|xs\|_{S}=\|x\|_{S}-1$. Then
	\begin{align*}
	\|g\|_{\std}&=\sum_{y\in \supp{f}}\|f(y)\|_{S_A}+2|[e_{F(S)},\supp{f}]\backslash[e_B,x]|+|[e_B,x]|\\
	&=\sum_{y\in \supp{f}}\|f(y)\|_{S_A}+2\left(|[e_{F(S)},\supp{f}]\backslash[e_B,xs]|-1\right)+|[e_B,xs]|+1\\
	&=\|gs\|_{\std}-1.
	\end{align*}
	
	In other words, $\|gs\|_{\std}=\|g\|_{\std}+1$ and we again contradict that $g$ is a dead end.

	Now let us prove the second part of the proposition. If $(A,S_A)$ has unbounded (retreat) depth, fix for $n\ge 1$ an element $a\in (A,S_A)$ of (retreat) depth $n$. Define the element $g=(f,e_{F(S)})\in A\wr F(S)$, where $f(x)=a$ if $\|x\|_S\le n$ and $f(x)=e_A$ otherwise. Then very similar arguments to the ones given above show that for any $s_1,\ldots,s_{n-1}\in S^{\pm 1},$
	
	$$
	\|gs_1\ldots s_{n-1}\|_{\std}=\|g\|_{\std}-(n-1).
	$$
	
	This implies that $g$ will be a dead end of (retreat) depth at least $n-1$ with respect to $\std$.	
\end{proof}
\section{The quasi-Hamiltonian property}\label{section:quasi-hamiltonian property}
Equation \eqref{eq: express word length of wreath product} tells us that in order to study word length of $(A\wr B,\std)$, we need to understand (at least partially) the solutions to the TSP in $\cay{B}{S_B}$. This problem is in general NP-hard, and hence we cannot hope to have a precise description of all solutions unless we are in very particular families of graphs (such as trees, which were studied in the previous section).

However, our focus on depth of lamplighter groups brings our attention onto very particular instances of the TSP. By looking at the structure of dead ends of lamplighters over trees given by Proposition \ref{prop: characterization of dead ends on lamplighters over free groups}, a naive approach to finding dead end elements in $A\wr B$ is to consider configurations of the form $g=(f,e_B)\in A\wr B$, where $f$ is a lamps configuration with support on a ball of radius $n$ centered at $e_B$. Hence, it is relevant for us to study the solutions to the TSP starting at $e_B$ and visiting all vertices in a big ball around the identity.

In Subsection \ref{subsection: qH motivation and examples} we show some examples of solutions to the TSP in some Cayley graphs, which serve as motivation for the rest of this section. Then in Subsection \ref{subsection: qH prop and unbounded depth} we define the quasi-Hamiltonian property for a graph and show that any group admits a quasi-Hamiltonian Cayley graph. As a consequence, any lamplighter group has a standard generating set with unbounded depth. Finally, in Subsection \ref{subsection: abelian groups} we prove that any Cayley graph of an abelian group, except for $\cay{\Z}{\{\pm 1\}}$ is quasi-Hamiltonian.  As a corollary, all standard generating sets of lamplighters over abelian groups have unbounded depth.
\subsection{Examples: solutions to the TSP inside some Cayley graphs}\label{subsection: qH motivation and examples}

\begin{example}\label{example: qH Z^2}
	Consider the group $\Z^2$ together with its canonical basis $\{\mathbf{e}_1,\mathbf{e}_2\}$ of $\Z^2$, and define the \textit{king's moves} generating set $S_{\mathrm{king}}=\{\pm\mathbf{e}_1,\pm\mathbf{e}_2,\mathbf{e}_1\pm\mathbf{e}_2,-\mathbf{e}_1\pm \mathbf{e}_2\}$. Then balls of $\cay{\Z^2}{S_{\mathrm{king}}}$ have the shape of squares, and it can be proved that the induced subgraphs $B_n\coloneqq B_{S_{\mathrm{king}}}(\mathbf{0},n)$ are Hamiltonian-connected for any $n\ge 1$. Hence for any $x\in B_n$ we have
	$$
	\TS{\mathbf{0}}{x}{B_n}=|B_n|+\delta_{x,\mathbf{0}},
	$$
	where $\delta_{\mathbf{0},\mathbf{0}}=1$ and $\delta_{x,\mathbf{0}}=0$ otherwise. Indeed, the shortest path from $\mathbf{0}$ to $x$ that covers $B_n$ visits each vertex exactly once, except possibly $\mathbf{0}$ which is visited twice if $x=\mathbf{0}$. Two such solutions are illustrated in Figure \ref{fig: Z^2 hamiltonian paths}.
	\begin{figure}[h!]
		\centering
		\begin{subfigure}[t]{0.5\textwidth}
			\centering
			\resizebox{0.95\textwidth}{!}{
				\begin{tikzpicture}
[
my star/.append style={star, draw, star points=5,minimum size=3mm,inner sep=0, star point ratio=2,scale=0.8, every node/.style={scale=0.8}}
]
\draw[step=0.5cm,black,very thin,opacity=0.2] (0,0) grid (10,10);
\draw[thick,->](0,5) -- (10,5);
\draw[thick,->](5,0) -- (5,10);
\node[below right, inner sep=7pt] at (8.5,5.5) {$n$};
\node[below left, inner sep=7pt] at (1.5,5.5) {$-n$};
\node[above left, inner sep=7pt] at (5.5,8.5) {$n$};
\node[below left, inner sep=7pt] at (5.5,1.5) {$-n$};
\begin{scope}[ultra thick,decoration={
	markings,
	mark=at position 0.01 with {\arrow{>}},
	mark=at position 0.05 with {\arrow{>}},
	mark=at position 0.15 with {\arrow{>}},
	mark=at position 0.26 with {\arrow{>}},
	mark=at position 0.37 with {\arrow{>}},
	mark=at position 0.46 with {\arrow{>}},
	mark=at position 0.57 with {\arrow{>}},
	mark=at position 0.665 with {\arrow{>}},
	mark=at position 0.76 with {\arrow{>}},
	mark=at position 0.84 with {\arrow{>}},
	mark=at position 0.94 with {\arrow{>}}}
] 
\draw[postaction={decorate},color=blue!50!black] (5,5) -- (5,8.5) -- (1.5,8.5) -- (1.5,1.5) 
-- (8.5,1.5) -- (8.5,2) -- (2,2) -- (2,8) -- (2.5,8) -- (2.5,2.5) -- (8.5,2.5) -- (8.5,3) -- (3,3) 
-- (3,8) -- (3.5,8) -- (3.5,3.5) -- (8.5,3.5) -- (8.5,4)-- (4,4) -- (4,8) -- (4.5, 8) -- (4.5,4.5)
-- (8.5,4.5) -- (8.5,5) -- (8.5,8.5) -- (5.5,8.5) -- (5.5,8) -- (8,8) -- (8,5) -- (7.5,5) 
-- (7.5,7.5) -- (5.5,7.5) -- (5.5,7) -- (7,7) -- (7,5) -- (6.5,5) -- (6.5,6.5) -- (5.5,6.5)  -- (5.5,6) -- (6,6)
-- (6,5) -- (5.5,5)  -- (5.5,5.5) -- (5,5) ; 
\fill (5,5) circle[radius=3pt];

\end{scope}
\end{tikzpicture}}
			\caption{Ending at $p=(0,0)$.}
		\end{subfigure}%
		~ 
		\begin{subfigure}[t]{0.5\textwidth}
			\centering
			\resizebox{0.95\textwidth}{!}{
				\begin{tikzpicture}
[
my star/.append style={star, draw, star points=5,minimum size=3mm,inner sep=0, star point ratio=2}
]
\draw[step=0.5cm,black,very thin,opacity=0.2] (0,0) grid (10,10);
\draw[thick,->](0,5) -- (10,5);
\draw[thick,->](5,0) -- (5,10);
\node[below right, inner sep=7pt] at (8.5,5.5) {$n$};
\node[below left, inner sep=7pt] at (1.5,5.5) {$-n$};
\node[above left, inner sep=7pt] at (5.5,8.5) {$n$};
\node[below left, inner sep=7pt] at (5.5,1.5) {$-n$};
\begin{scope}[ultra thick,decoration={
	markings,
	mark=at position 0.01 with {\arrow{>}},
	mark=at position 0.05 with {\arrow{>}},
	mark=at position 0.15 with {\arrow{>}},
	mark=at position 0.26 with {\arrow{>}},
	mark=at position 0.37 with {\arrow{>}},
	mark=at position 0.46 with {\arrow{>}},
	mark=at position 0.57 with {\arrow{>}},
	mark=at position 0.665 with {\arrow{>}},
	mark=at position 0.76 with {\arrow{>}},
	mark=at position 0.84 with {\arrow{>}},
	mark=at position 0.94 with {\arrow{>}}}
] 
\draw[postaction={decorate},color=blue!50!black] (5,5) -- (5,8.5) -- (1.5,8.5) -- (1.5,1.5) 
-- (8.5,1.5) -- (8.5,2) -- (2,2) -- (2,8) -- (2.5,8) -- (2.5,2.5) -- (8.5,2.5) -- (8.5,3) -- (3,3) 
-- (3,8) -- (3.5,8) -- (3.5,3.5) -- (8.5,3.5) -- (8.5,4)-- (4,4) -- (4,8) -- (4.5, 8) -- (4.5,4.5)
-- (8.5,4.5) -- (8.5,5) -- (8.5,8.5) -- (5.5,8.5) -- (5.5,8) -- (8,8) -- (8,5) -- (7.5,5) -- (5.5,5)
-- (5.5,7.5) -- (6,7.5) -- (6,5.5) -- (7.5,5.5) -- (7.5,6) -- (6.5,6) -- (6.5,6.5) -- (7.5,6.5)
-- (7.5,7.5) -- (6.5,7.5) -- (6.5,7) -- (7,7);
\fill (5,5) circle[radius=3pt];
\fill (7,7) circle[radius=3pt];

\end{scope}
\end{tikzpicture}}
			\caption{Ending at $\mathbf{p}=(4,4)$.}
		\end{subfigure}
		\caption{Paths visiting all vertices in the square $[-n,n]^2$, starting at $(0,0)$ and finishing inside the square.}
		\label{fig: Z^2 hamiltonian paths}
	\end{figure}
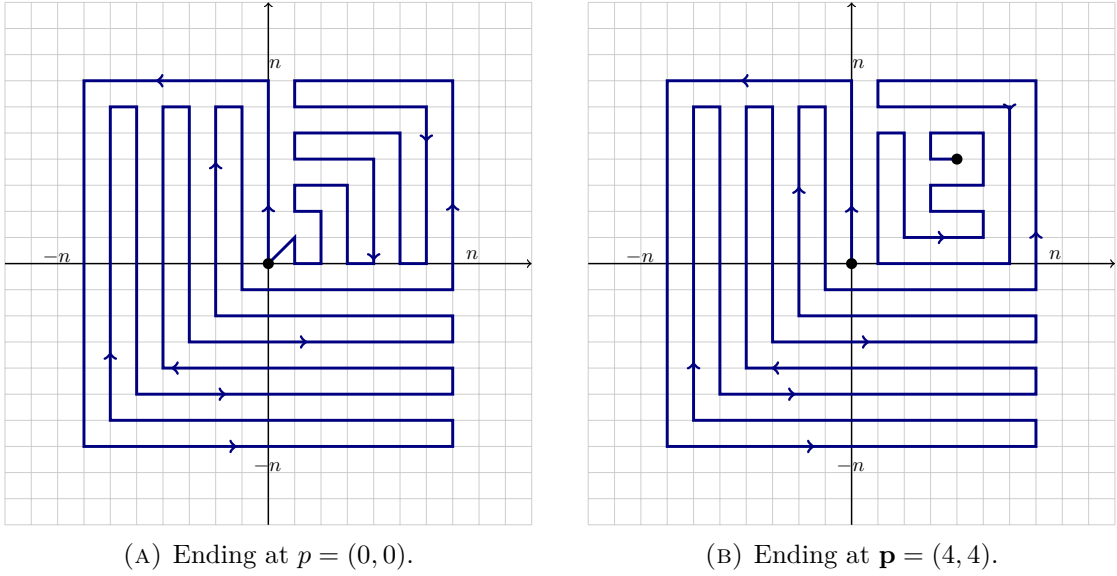
\end{example}

\begin{example}\label{example: qH Z with {1,2}}
	Now we consider $\Z$ with the generating set $S=\{\pm1 ,\pm2 \}$. The Cayley graph $\cay{\Z}{S}$ is illustrated in Figure \ref{fig: Cayleygraph of (Z,{1,2}) is not Hamiltonian-connected}. 	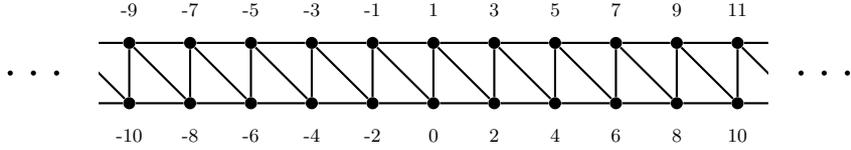
\begin{figure}[h!]
		\centering
		\begin{tikzpicture}[dot/.style={circle,inner sep=2pt,fill},scale=0.8, every node/.style={scale=0.8}]
\node at (-6.5,0.5) {\Huge $\dots$};
\node at (6.5,0.5) {\Huge $\dots$};
\clip (-5.5,-1) rectangle (5.5,2);
\foreach \x in {0,2,4,6,8,10,12}{
	\node[dot,label={[label distance=0.2cm]-90:\footnotesize \x}] (\x) at (\x*0.5,0) {};
}
\foreach \x in {-2,-4,-6,-8,-10,-12}{
	\node[dot,label={[label distance=0.2cm]-90:\footnotesize \x}] (\x) at (\x*0.5,0) {};
}
\foreach \x[evaluate=\x as \y using {int(\x+1)}] in {0,2,4,6,8,10,12}{
	\node[dot,label={[label distance=0.2cm]90:\footnotesize \y}] (\x+1) at (\x*0.5,1) {};
}
\foreach \x[evaluate=\x as \y using {int(\x+1)}] in {-2,-4,-6,-8,-10,-12}{
	\node[dot,label={[label distance=0.2cm]90:\footnotesize \y}] (\x+1) at (\x*0.5,1) {};
}

\foreach \x [evaluate=\x as \y using {int(\x+2)}] in {-12,-10,...,10}{
	\draw[thick,-](\x) -- (\y);
	\draw[thick,-] (\x) -- (\x+1);
}
	\draw[thick,-] (12) -- (12+1);
\foreach \x[evaluate=\x as \y using {int(\x+2)}] in {-12,-10,...,10}{
	\draw[thick,-](\x+1) -- (\y+1);
}
\foreach \x[evaluate=\x as \y using {int(\x+2)}] in {-12,-10,...,10}{
	\draw[thick,-](\y) -- (\x+1);
}

\end{tikzpicture}
		\caption{The Cayley graph $\cay{\Z}{\{\pm 1,\pm 2\}}$.}
		\label{fig: Cayleygraph of (Z,{1,2}) is not Hamiltonian-connected}
	\end{figure}

	The ball $B_n$ of radius $n$ centered at $0$ is not Hamiltonian-connected, since a path going from $0$ to $1$ and visiting all vertices of $B_n$ must visit one vertex twice. However, it does hold that
	$$
	|B_n|\le \TS{0}{x}{B_n}\le |B_n|+1,
	$$
	for any $x\in B_n$. Hence, even though the graph is not Hamiltonian-connected, solutions to the TSP starting at $0$ and visiting all of $B_n$ differ from a hypothetical Hamiltonian path only by a uniform additive constant.
	
\end{example}

\begin{example}\label{example: qH free group}
	Now consider a finite non-empty set $S$ and the free group $F(S)$. In this case the Cayley graph $\cay{F(S)}{S}$ is a tree and for any $x\in B_n\coloneqq B_S(e_{F(S)},n)$,
	$$
	\TS{e_{F(S)}}{x}{B_n}\ge n+|B_n|.
	$$
	Indeed, a path visiting all vertices in the ball $B_n$ must pass through $e_{F(S)}$ at least twice, and as $\cay{F(S)}{S}$ is a tree this implies that it traverses a geodesic from $e_{F(S)}$ to an element of word length $n$ twice. 
\end{example}

In the first two examples, minimal spanning paths of $B_n$ repeat a constant number of vertices, while on the third one it is necessary to repeat an unbounded number of them. On what follows we study these different behaviors and their consequences of depth properties of lamplighter groups.

\subsection{The quasi-Hamiltonian property and unbounded depth of lamplighters}\label{subsection: qH prop and unbounded depth}
We start by defining the property of $\cay{B}{S_B}$ that will be a sufficient condition for the existence of dead ends of unbounded depth in $A\wr B$, and which we believe to be of interest on its own.

\begin{definition}[Quasi-Hamiltonian property]\label{def: qH property} Let $\Gamma$ be an infinite, connected and locally finite graph, and fix a vertex $o\in V$. Denote by $B_n$ the ball of $\Gamma$ centered at $o$ of radius $n$. We say that $(\Gamma,o)$ has the \textit{quasi-Hamiltonian property} or that it is \textit{quasi-Hamiltonian} if there exists a family $\mathcal{F}$ of connected (induced) finite subgraphs of $\Gamma$ such that
	\begin{enumerate}
		\item $o\in F$ for every $F\in \mathcal{F}$,
		\item for any $n\ge 1$, there exists $F\in \mathcal{F}$ such that $B_n\subseteq F$, and
		\item there exists a constant $M\ge 0$ such that for any $F\in \mathcal{F}$ and $x\in F$,
		$$
		\TS{o}{x}{F}\le  |F|+M.
		$$
	\end{enumerate}
\end{definition}
We concentrate on the case where $\Gamma=\cay{B}{S_B}$, and $o=e_B$. 

Note that Definition \ref{def: qH property} implies that for any $F\in \mathcal{F}$ and $x\in F$,
$$
|F|\le\TS{e_B}{x}{F}\le |F|+M,
$$
so that the lengths of optimal paths are at bounded distance from those of hypothetical Hamiltonian paths.

\begin{lemma}\label{lem: Hamiltonian sequence implies unbounded depth}
	If $\cay{B}{S_B}$ is quasi-Hamiltonian, then for every group $(A,S_A)$ of unbounded depth the corresponding lamplighter group $\displaystyle\left (A \wr B,\std \right)$ has unbounded depth.
\end{lemma}
\begin{proof}
	
	Since $\cay{B}{S_B}$ has the quasi-Hamiltonian property, there exists $M\ge 0$ such that for any $n\ge 1$ we can find $F\subseteq B$ a connected subgraph with $B_{S_B}(e_B,n)\subseteq F$ and with 
	$$
	|F|\le\TS{e_B}{p}{F}\le |F|+M,
	$$
	for any $p\in F$.
	
	Fix a dead end $a\in A$ of depth at least $n$ with respect to $S_A$, and consider the configuration $g=(f,e_B)\in A \wr B$, where $f(x)=a$ if $x\in F$ and $f(x)=e_A$ otherwise. We see that $$\|g\|_{\std}\ge \sum_{a\in F}\|a\|_{S_A} +|F|,$$ since in order to light all lamps at vertices of $F$ it is mandatory to visit each of these elements at least once.
	
	Now consider any element $h\in B_{\std}(e,n)$, and note that the element $gh$ corresponds to a new lamplighter configuration where some of the lamps at positions of $F$ may have changed (by at most $n$ generators), and the new position of the lamplighter is some element $p\in B_{S_B}(e_B,n)\subseteq F$. Such an element can be constructed using the generators of $\std$ by following a spanning path of $F$ starting at $e_B$ and finishing at $p$, while generating the states of the lamps at configurations of $F$. We see hence that
	$$
	\|gh\|_{\std}\le \sum_{a\in F}\|a\|_{S_A}+ |F|+M \le \|g\|_{\std}+M.
	$$
	This inequality together with the Fuzz Lemma \ref{lem: Fuzz Lemma} prove the existence of a dead end of depth at least $n/M$. As $M$ does not depend on $n$, we conclude that $(A\wr B,\std)$ has dead ends of unbounded depth.
\end{proof}

\begin{remark}\label{remark: super generating sets and hamiltonian sequences} Suppose we have two generating sets $S,S'$ of $B$, with $S\subseteq S'$. By noting that $\cay{B}{S'}$ can be obtained from $\cay{B}{S}$ by adding a finite number of extra edges at each vertex, we see that if $(B,S)$ has the quasi-Hamiltonian property then so does $(B,S').$ 
	
	More generally, if $\Gamma$ is any subgraph obtained from $\cay{B}{S}$ by removing edges (but not vertices) and $\Gamma$ has the quasi-Hamiltonian property, then so does $\cay{B}{S}$.
\end{remark}

\subsection{Existence of standard generating sets with unbounded depth}
Now we prove that any infinite group admits a quasi-Hamiltonian Cayley graph. In order to do so, we use Lemma \ref{lem: the cube of a connected graph is Hamiltonian connected}, which tells us that the cube of any connected finite graph is Hamiltonian-connected. This result has been used in a similar manner by Georgakopoulos in order to prove the existence of generating sets with Hamiltonian circles in Cayley graphs on infinite groups \cite{GeorgakopoulosHamiltonian2009}, and by Khukhro \cite{Khukhro2020} and Ostrovskii and Rosenthal \cite{OstrovskiiRosenthal2015} in order to show that non virtually free groups admit Cayley graphs which have any finite graph as a minor.

\begin{lemma}[Existence of a quasi-Hamiltonian Cayley graphs]\label{lem: existence of genset with quasiHamiltonian sequence} Any finitely generated group $B$ admits a quasi-Hamiltonian Cayley graph.
\end{lemma}
\begin{proof}
	Start with any symmetric finite generating set $S$ for $B$, and for an arbitrary $n\ge 1$ consider $F=B_S(e_B,3n)$. Thanks to Lemma \ref{lem: the cube of a connected graph is Hamiltonian connected}, the cube of the induced graph by $F$ in $\cay{B}{S}$ is Hamiltonian-connected, and hence $F$ is a Hamiltonian-connected subset of $\cay{B}{S\cup S^2 \cup S^3}$.
	
	Denote this new generating set $S_B\coloneqq S\cup S^2 \cup S^3$. By definition, $F=B_{S_B}(e_B,n)$ is a Hamiltonian-connected subgraph. 
	
	As $n$ was arbitrary, this proves that $\cay{B}{S_B}$ has the quasi-Hamiltonian property.
\end{proof}

\begin{thm}\label{thm: existence of genset with deep dead ends in any lamplighter} Let $A$ be a finitely generated group with unbounded depth for some finite generating set, and $B$ be any finitely generated group. Then there exists a standard generating set of $A\wr B$ with unbounded depth.
\end{thm}
\begin{proof}
	Thanks to Lemma \ref{lem: Hamiltonian sequence implies unbounded depth}, the result follows from the existence of a quasi-Hamiltonian presentation given by Lemma \ref{lem: existence of genset with quasiHamiltonian sequence}.
\end{proof}

\begin{example} Let $S$ be a finite set with at least two elements, and consider the free group $F(S)$. It follows from Example \ref{example: qH free group} that $\cay{F(S)}{S}$ does not have the quasi-Hamiltonian property, and from Lemma \ref{lem: existence of genset with quasiHamiltonian sequence} that $\cay{F(S)}{S\cup S^2 \cup S^3}$ does. It is natural to ask about what happens for $\cay{F(S)}{S\cup S^2}$, and the answer is that it does not have the quasi-Hamiltonian property.
	
	Indeed, let $H$ be a subset of $F(S)$ that contains a ball of radius $n$, and let $\gamma$ be a spanning cycle of $H$ of length $\TS{e_{F(S)}}{e_{F(S)}}{H}$. Doing a case by case analysis, it is possible to see that for any element $v\in F(S)$ with $\|v\|_S\le n$, the path $\gamma$ must repeat at least one vertex in its $1$-neighborhood. This means that  $\TS{e_{F(S)}}{e_{F(S)}}{H}-|H| \xrightarrow[n \to +\infty]{} +\infty$, and so $(F(S),S\cup S^2)$ is not a quasi-Hamiltonian presentation.
	
	For this example to work, it is essential that $|S|\ge 2$. As we showed in Example \ref{example: qH Z with {1,2}}, $\cay{\Z}{\{1,2\}}$ does have the quasi-Hamiltonian property.
\end{example}

\subsection{Abelian groups}\label{subsection: abelian groups}

In this subsection, we prove that any Cayley graph of a finitely generated abelian group different from $\cay{\Z}{\{\pm 1\}}$ is quasi-Hamiltonian. Since the lamplighter group over $\cay{\Z}{\{\pm 1\}}$ is covered by Cleary and Taback's original example in \cite{ClearyTaback05}, we conclude that any lamplighter group $A\wr B$ over an abelian base group $B$ has unbounded depth, with respect to every standard generating set $\std=S_A\cup S_B$, as long as the lamps group $(A,S_A)$ has unbounded depth.

We begin by sketching the proof of the fact that every Cayley graph of an abelian group other than $\cay{\Z}{\pm 1}$ has the quasi-Hamiltonian property. Our starting point are ``grid graphs'', that is, graphs whose vertex set is $$\{1,\ldots,n\}\times \{1,\ldots,m\}, \text{ for } n,m\ge 2,$$ and where edges connect vertices of the form $(i,j)$ with $(k,l)$ if and only if $|i-k|+|j-l|=1$, for $1\le i,k\le n$ and $1\le j,l\le m$. Hamiltonian paths in such graph have been studied by Itai, Papadimitiou and Szwarcfiter \cite{ItaiPapadimitriouSzwarcfiter1982}, and their results imply that between any two vertices there is a spanning path that repeats at most $2$ elements. This shows that the Cayley graph of $\Z^2$ with standard generators, as well as the graph $\Z\times \{1,\ldots, n\}$, $n\ge 2$, have the quasi-Hamiltonian property. The rest of the proof consists of showing that any Cayley graph of an abelian group $G$, other than $\cay{\Z}{\{\pm 1\}}$, contains one of the above graphs as a spanning subgraph. In other words, that for any Cayley graph $\cay{G}{S}$ as above, there exists a bijective $1$-Lipschitz embedding whose domain is one of the graphs $\Z^2$ or $\Z\times \{1,\ldots,n\}$, for some $n\ge 2$. This will follow from an inductive argument, based on the proofs of Vászonyi \cite{Vazsonyi1939} and Nash-Williams \cite{NashWilliams1959} of the existence of Hamiltonian double rays in the Cayley graph of any abelian group.

In order to formulate our results, we introduce some notation. Given $m\ge 1$, we denote by $I_{m}$ the interval graph on $m$ vertices. That is, the graph whose vertex set is $\{1,\ldots,m\}$ and where edges connect $i$ with $i+1$, for $0\le i<m$. More generally, for any integers $m_1,\ldots,m_s\ge 1$, we use the notation

\[\mathrm{Cube}(m_1,\ldots,m_s)\coloneqq I_{m_1}\times \cdots\times I_{m_s}\]

for the product graph of all the $I_{m_i}$'s (see Definition \ref{def: product graph}). When $s=2$, we call $\mathrm{Cube}(m_1,m_2)$ a \textit{grid graph}.

The existence of Hamiltonian paths between two vertices of a grid graph $\mathrm{Cube}(m_1,m_2)$ is studied in \cite{ItaiPapadimitriouSzwarcfiter1982}. In particular, it is shown that obstructions to Hamiltonian-connectedness arise either from a parity issue, or by some particular configurations when either $m_1$ or $m_2$ are at most $3$. Examples of such non-Hamiltonian-connected grid graphs are illustrated in Figure \ref{fig: grid graphs pathological}.
\begin{figure*}[h!]
	\centering
	\begin{subfigure}[t]{0.4\textwidth}
		\centering
		\resizebox{1\textwidth}{!}{
			\begin{tikzpicture}[dot/.style={circle,inner sep=1.5pt,fill}]

\foreach \x in {0,1,...,5}{
	\foreach \a in {0,1,2}{
	\node[dot] (\x\a) at (\x,\a) {};
}}
\foreach \x [evaluate=\x as \y using {int(\x+1)}] in {0,1,...,4}{
	\foreach \a [evaluate=\a as \b using {int(\a+1)}] in {0,1}{
	\draw[thick,-](\x\a) -- (\x\b);
	\draw[thick,-](\x\a) -- (\y\a);
}
}
\foreach \x [evaluate=\x as \y using {int(\x+1)}] in {0,1,...,4}{
		\draw[thick,-](\x2) -- (\y2);
}
\foreach \a [evaluate=\a as \b using {int(\a+1)}] in {0,1}{
		\draw[thick,-](5\a) -- (5\b);
	}
\node at ([shift={(0.2,0.2)}]10) {$u$};
\node at ([shift={(0.2,0.2)}]41) {$v$};
\end{tikzpicture}}
		\caption{}
	\end{subfigure}%
	~ 
	\begin{subfigure}[t]{0.4\textwidth}
		\centering
		\resizebox{1\textwidth}{!}{
			\begin{tikzpicture}[dot/.style={circle,inner sep=1.5pt,fill}]

\foreach \x in {0,1,...,5}{
	\foreach \a in {0,1,2}{
		\node[dot] (\x\a) at (\x,\a) {};
}}
\foreach \x [evaluate=\x as \y using {int(\x+1)}] in {0,1,...,4}{
	\foreach \a [evaluate=\a as \b using {int(\a+1)}] in {0,1}{
		\draw[thick,-](\x\a) -- (\x\b);
		\draw[thick,-](\x\a) -- (\y\a);
	}
}
\foreach \x [evaluate=\x as \y using {int(\x+1)}] in {0,1,...,4}{
	\draw[thick,-](\x2) -- (\y2);
}
\foreach \a [evaluate=\a as \b using {int(\a+1)}] in {0,1}{
	\draw[thick,-](5\a) -- (5\b);
}
\node at ([shift={(0.2,0.2)}]21) {$u$};
\node at ([shift={(0.2,0.2)}]31) {$v$};
\end{tikzpicture}}
		\caption{}
	\end{subfigure}
	\caption{There are no Hamiltonian paths from $u$ to $v$ in these grid graphs.}
	\label{fig: grid graphs pathological}
\end{figure*}
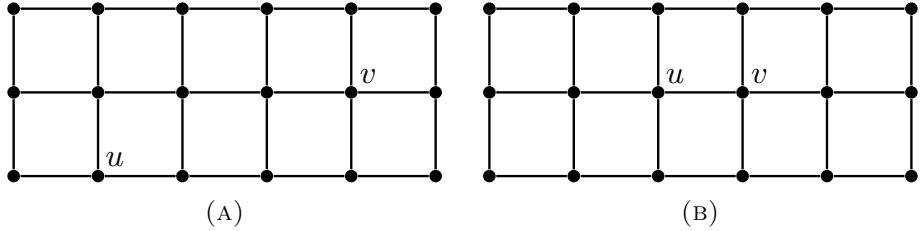

\begin{lemma}[Hamiltonian-connectivity of grid graphs]\label{lem: Itai grid graphs} Consider $m_1,m_2\ge 2$. Then between any two vertices $s,t\in \mathrm{Cube}(m_1,m_2)$, there exists a spanning path of $\mathrm{Cube}(m_1,m_2)$ of length at most $|\mathrm{Cube}(m_1,m_2)|+2$.
\end{lemma}
\begin{proof}
	A necessary and sufficient condition for the existence of a Hamiltonian path of $\mathrm{Cube}(m_1,m_2)$ between two vertices is provided in \cite[Theorem 3.2]{ItaiPapadimitriouSzwarcfiter1982}, which depends only on whether the grid graph is bipartite, and on some particular cases where $m_1\le 3$ or $m_2\le 3$. This result implies that for any two pair of vertices $s,t\in \mathrm{Cube}(m_1,m_2)$, there is either a Hamiltonian path from $s$ to $t$, or a Hamiltonian path from $s$ to a vertex at distance at most $2$ from $t$.
\end{proof}

Now we prove the higher-dimensional version of Lemma \ref{lem: Itai grid graphs}.

\begin{lemma}\label{lem: rectangles are quasi-Hamiltonian-connected} For any $r,m_1,\ldots,m_s\ge 1$, consider the graph 
	$$\Gamma=\Z^r\times \mathrm{Cube}(m_1,\ldots,m_s),$$
	where $\Z^r$ is identified with its Cayley graph with respect to canonical generators.

	Suppose that $\Gamma$ is not a line, that is, either $r\ge 2$, or $r=1$ and $s\ge 1$. Then there exists a constant $M\ge 0$ such that for any $n_1,\ldots,n_r\ge 1$, the induced subgraph
	$$
	R=\mathrm{Cube}(n_1,\ldots,n_r)\times  \mathrm{Cube}(m_1,\ldots,m_s)\subseteq \Z^r\times \mathrm{Cube}(m_1,\ldots,m_s)
	$$
	has a spanning path of length at most $|R|+M$ between any two vertices $s,t\in R$.
\end{lemma}
\begin{proof}
	We start by noting that if $r=s=1$ then $\Gamma\cong\Z\times I_{m_1}$, and that if $r=2$ and $s=0$ then $\Gamma\cong\Z^2$. In both cases $R$ is isomorphic to a grid graph of appropriate dimensions, and hence the result follows from Lemma \ref{lem: Itai grid graphs} with $M=2$. These are the base cases for an inductive argument, which we explain now.

	Let us first consider the case $r=1$ and suppose that $s\ge 2$,  so that we have
	$$
	\Gamma\cong \Z\times \mathrm{Cube}(m_1,\ldots,m_s), 
	$$
	
	For $\mathrm{Cube}(m_{s-1},m_s)$ consider the Hamiltonian path $P$ that starts at $(1,1)$, traverses the edges of $I_{m_{s-1}}\times \{1\}$ until it reaches $(m_{s-1},1)$, then crosses to $(m_{s-1},2)$ and continues in a similar way traversing one copy of $I_{m_{s-1}}$ at the time, until it finally reaches either $(m_{s-1},m_s)$ or $(m_{s-1},1)$, after having visited all vertices of the graph $I_{m_{s-1}}\times I_{m_s}$.

	Define the function $$h:\mathrm{Cube}(m_{s-1},m_s)\to \{1,\ldots, m_{s-1}\cdot m_s\}$$ which assigns to each pair $(j,k)\in  \mathrm{Cube}(m_{s-1},m_s)$ its unique position $$h(j,k)\in \{1,\ldots, m_{s-1}\cdot m_s\}$$ in the path $P$. That is, if we write $P=P_0,P_1,\ldots,P_{\ell}$ as a sequence of vertices, then $$
	(j,k)=P_{h(j,k)}, \ \text{  for  }(j,k)\in  \mathrm{Cube}(m_{s-1},m_s).
	$$
	
	Note that as $P$ is a Hamiltonian path, the function $h$ is well defined and bijective.
	
	With the above, we can see that any subgraph $$R=I_{n}\times \mathrm{Cube}(m_1,\ldots,m_s)$$
	
	has a spanning subgraph isomorphic to another one of same shape in the graph $$\Z\times \mathrm{Cube}(m_1,\ldots,m_{s-2}, m_{s-1}m_s).$$ Indeed, it suffices to use the function $h$ to map any element $(i,j_1,j_2,\ldots,j_s)\in R$ into $$\big(i,j_1,\ldots,j_{s-2},h(j_{s-1},j_s)\big)\in \Z\times  \mathrm{Cube}(m_1,\ldots,m_{s-2}, m_{s-1}m_s).$$ This construction is illustrated in Figure \ref{fig: inductive step abelian groups} for $s=2$.

	\begin{figure*}[h!]
		\centering
		\begin{tikzpicture}[dot/.style={circle,inner sep=1pt,fill},scale=0.8, every node/.style={scale=0.8}]
\def\angl{50}
\def\angldepth{180}
\def\tot{16}

\node[dot] (A0) at (0,0) {};
\node[dot] (B0) at ([shift=({90:2.5 cm})]A0) {};
\node[dot] (C0) at ([shift=({\angl:2.5 cm})]B0) {};
\node[dot] (D0) at ([shift=({\angl:2.5 cm})]A0) {};
\node (m1) at ([shift=({\angldepth:\tot*0.66 cm})]A0) {$m_{1}$};
\node (m2) at ([shift=({\angl:3.1 cm})]A0) {$m_{2}$};
\node (Z) at ([shift=({90:5 cm})]A0) {$\mathbb{Z}$};
\draw[thick,black,->] (A0) -- (m1);
\draw[thick,black,->] (A0) -- (m2);
\draw[thick,black,->] (A0) -- (Z);
\draw[thick,blue!70!black] (A0) -- (B0) -- (C0) -- (D0) -- (A0);

\foreach \x in {1,2,...,\tot}{
	\node[dot] (A\x) at ([shift=({\angldepth:\x*0.6 cm})]A0) {};
	\node[dot] (B\x) at ([shift=({\angldepth:\x*0.6 cm})]B0) {};
	\node[dot] (C\x) at ([shift=({\angldepth:\x*0.6 cm})]C0) {};
	\draw[thick,blue!70!black] (A\x) -- (B\x) -- (C\x);
}
\foreach \x[evaluate= \x as \y using {int(\x-1)}] in {1,3,...,\tot}{
	\draw[thick,blue!70!black] (A\x) -- (A\y);
		\draw[thick,blue!70!black] (B\x) -- (B\y);
}
\foreach \x[evaluate= \x as \y using {int(\x-1)}] in {2,4,...,\tot}{
	\draw[thick,blue!70!black] (C\x) -- (C\y);
}
\foreach \x in {1,2,...,\tot}{
\node (w) at ([shift=({-90:0.84 cm})]C\x) {};	
\draw[thick,blue!70!black] (C\x) -- (w);
}

\foreach \x in {2,4,...,\tot}{
	\node (w) at ([shift=({\angl:1.085 cm})]A\x) {};	
	\draw[thick,blue!70!black] (A\x) -- (w);
}
\end{tikzpicture}
		\caption{Inductive step of the proof of Lemma \ref{lem: rectangles are quasi-Hamiltonian-connected}.}
		\label{fig: inductive step abelian groups}
	\end{figure*}
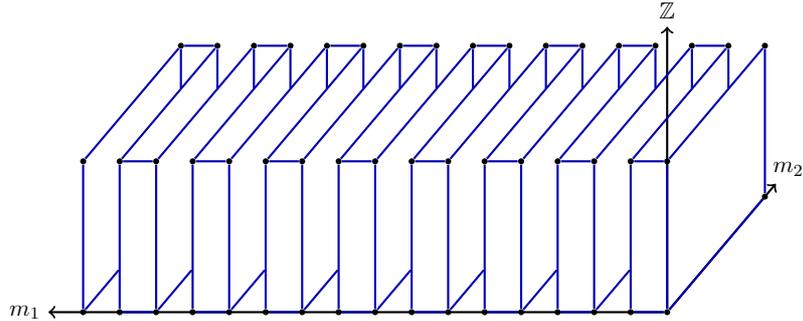
	
	Then, thanks to the induction hypothesis, this new subgraph has spanning paths of length at most $|R|+M$ between any pair of vertices, and so the same holds for $R$. This concludes the induction for the case $r=1$.
	
	For the remaining cases $r\ge 2$, a very similar induction proves the result. Indeed, now we repeat the argument using a Hamiltonian path between opposite corners of a grid subgraph, in order to find a spanning subgraph which is (isomorphic to) a similar subgraph on a graph with a lower number of finite factors $s$, or lower free rank $r$. Together with the already proved base case of $r=2$ and $s=0$, this finishes the inductive argument.
\end{proof}

The above implies that for any $r\ge 1$ and $m_1,\ldots, m_s\ge 1$ as in the hypotheses of Lemma \ref{lem: rectangles are quasi-Hamiltonian-connected}, the Cayley graph of 
$$
\Z^r\times \Z/m_1\Z\times\cdots \Z/m_s\Z
$$
with standard generators has the quasi-Hamiltonian property. In order to generalize this to hold for \textit{every} Cayley graph, we will use the following result about the structure of generating sets of abelian groups.

\begin{lemma}[{Nash-Williams, \cite[Lemma 3]{NashWilliams1959}}]\label{lem: Nash williams gensets of abelian groups} Let $B$ be an infinite abelian group with free abelian rank $\mathrm{rank}(B)=r\ge 1$, and $S$ any finite generating set. Then there is an enumeration $a_1,\ldots,a_r,b_1,\ldots,b_s$ of the elements of $S$ and there exist positive integers $m_1,\ldots,m_s$, such that each element $g\in B$ is uniquely expressible as
	$$
	g=\sum_{i=1}^{r}p_i a_i + \sum_{j=1}^s q_jb_j,
	$$ 
	where the $p_i,q_j\in \Z$ and for each $j=1,\ldots,s$, we have $0\le q_j<m_j$.
\end{lemma}

\begin{lemma}\label{lem: reducing any abelian generators to the canonical case}
	Under the assumptions of Lemma \ref{lem: Nash williams gensets of abelian groups}, the function
	\begin{align*}
	\varphi:\Z^r\times  \mathrm{Cube}(m_1,\ldots,m_s)&\to B \\
	(p_1,p_2,\ldots,p_r,q_1,q_2,\ldots,q_s)&\mapsto \sum_{i=1}^r p_i a_i + \sum_{j=1}^s q_j b_j
	\end{align*}
	is a bijective $1$-Lipschitz embedding of $\Z^r\times  \mathrm{Cube}(m_1,\ldots,m_s)$ onto the Cayley graph $\cay{B}{S}$. Moreover, $r$ can be chosen to be any integer between $\min\{2,\mathrm{rank}(B)\}$ and $\mathrm{rank}(B)$.
\end{lemma}
\begin{proof}
	
	The first sentence follows from Lemma \ref{lem: Nash williams gensets of abelian groups}, while the second one follows from an analogous inductive argument to the one used in the proof of Lemma \ref{lem: rectangles are quasi-Hamiltonian-connected}.
\end{proof}

\begin{proposition}\label{prop: any cayley graph of a fg abelian group has a quasi Hamiltonian sequence} With the exception of $\cay{\Z}{\{\pm 1\}}$, any Cayley graph of an infinite finitely generated abelian group is quasi-Hamiltonian.
\end{proposition}
\begin{proof}
	Let $(B,S)$ be an infinite finitely generated abelian group, different from $(\Z,\{\pm 1\})$. Then Lemma \ref{lem: reducing any abelian generators to the canonical case} tells us that for some $r,m_1,\ldots,m_s\ge 1$, there is a bijective $1$-Lipschitz embedding of $\Z^r\times \mathrm{Cube}(m_1,\ldots,m_s)$ onto $\cay{B}{S}$. In other words, tha former graph is a spanning subgraph of the latter one.
	
	Now using Lemma \ref{lem: rectangles are quasi-Hamiltonian-connected}, there exists a constant $M\ge 0$ such that any rectangle
	$$
	R_n=[-n,n]^r\times \mathrm{Cube}(m_1,\ldots,m_s),
	$$
	for $n\ge 1$, has spanning paths of length at most $|R_n|+M$ between any pair of vertices. Any ball of $\cay{B}{S}$ centered at the origin is contained in $R_n$ for $n$ sufficiently large, and the above implies that we can find paths from the origin to any other vertex with paths of length at most $|R_n|+M$. This shows that $\cay{B}{S}$ has the quasi-Hamiltonian property.
\end{proof}

\begin{proposition}\label{prop: any lamplighter over an abelian group has unbounded depth wrt standard gensets} For $(A,S_A)$ a group of unbounded depth, and $(B,S_B)$ any finitely generated abelian group, the lamplighter group $A\wr B$ has unbounded depth with respect to $\std$.
\end{proposition}
\begin{proof}
	The case where $(B,S_B)=(\mathbb{Z},{\pm 1})$ follows from the original example of Cleary and Taback \cite{ClearyTaback05}, or alternatively it can be seen as a particular case of Proposition \ref{prop: characterization of dead ends on lamplighters over free groups} for lamplighters over trees.
	
	In any other case, Corollary \ref{prop: any cayley graph of a fg abelian group has a quasi Hamiltonian sequence} shows that $(B,S_B)$ has the quasi-Hamiltonian property and hence Lemma \ref{lem: Hamiltonian sequence implies unbounded depth} implies that the group $(A\wr B,\std)$ has unbounded depth.
\end{proof}
A natural question is whether the claim of Lemma \ref{lem: reducing any abelian generators to the canonical case} holds for some non-abelian groups. That is, whether there is a bijective $1$-Lipschitz embedding of a graph of the form $\Z^r\times \mathrm{Cube}(m_1,\ldots,m_s)$ onto all of its Cayley graphs. As shown by Nash-Williams, any such graph must admit Hamiltonian double ray. With respect to this property, Thomassen has shown that it holds for any Cayley graph of the form $\cay{G}{S\cup S^2}$, where $G$ is a $1$-ended group and $S$ is a finite generating set \cite{Thomassen1978}. In \cite{Seward2014}, Seward studies \textit{translation-like actions} of free groups. In particular, they characterize groups with finitely many ends as those which admit a transitive translation-like action of $\Z$, and show that this is equivalent to having a Cayley graph that admits a Hamiltonian double ray. In the same paper, it is mentioned in Problem 4.8 that the existence of a Hamiltonian double ray in \textit{every} Cayley graph of an infinite finitely generated group with finitely many ends is an open question. With respect to our question of extending Lemma \ref{lem: reducing any abelian generators to the canonical case}, a first approach could be to study transitive translation-like actions of $\Z^r$, $r\ge 2$.
\subsection{Remarks}
Unlike the original example of Cleary and Taback, the dead ends of $A\wr B$ we found using quasi-Hamiltonian Cayley graphs of the base group $B$ are of bounded retreat depth (Definition \ref{def: retreat depth}). Hence the following question remains open.
\begin{question}
	Suppose $(A,S_A)$ has unbounded retreat depth. Does every lamplighter group $A\wr B$ admit a standard generating set with unbounded retreat depth?
\end{question}

So far we have seen that this holds for $B$ a free group with a free generating set (Proposition \ref{prop: characterization of dead ends on lamplighters over free groups}), but to our knowledge there are no other known examples. An answer to this question would be of interest even in the case of particular base groups, as for example $B=\Z\times \Z/2\Z$ or $B=\Z^2$.

Another observation is that in the groups we have studied so far, the constant $M$ from the definition of the quasi-Hamiltonian property is at most $2$. We have not been able to find examples of a Cayley graph for which this constant is necessarily bigger. More generally, we ask the following.
\begin{question}
	Given an arbitrary $n\ge 1$, does there exist a group $B$ together with a generating set $S_B$ such that $\cay{B}{S_B}$ satisfies Definition \ref{def: qH property} with $M=n$ but not with $M=n-1$?
\end{question}
\section{Lamplighters over free products of finite groups}\label{section: lamplighters over free products}
Our results in the previous sections concern lamplighter groups with unbounded depth with respect to standard generating sets. We begin this section by showing that this is not always the case, with an example of a lamplighter group over the free product $\Z/8\Z*\Z/2\Z$ that has uniformly bounded depth with respect to standard generators. 

Next, we characterize which standard generators of lamplighter groups over free products of finite groups have this property. For this, we define for a finite group $G$ with a generating set $S_G$ its \textit{Hamiltonian difference} $\mathscr{H}(G,S_G)$ (Definition \ref{def: new constant hamiltonicity finite groups}), which measures how much shorter minimal spanning cycles are than minimal spanning paths inside $\cay{G}{S_G}$. The main result of this section says that a lamplighter group over the free product $H*K$, where $(H,S_H)$ and $(K,S_K)$ are finite groups with their respective generating sets, has uniformly bounded depth if and only if $\mathscr{H}(H,S_H)+\mathscr{H}(K,S_K)\ge 1$ (Theorem \ref{thm: general dead ends on free products}).

In what follows we use the following observation. For $H,K$ finite groups, consider their free product $H*K$ with a generating set of the form $S_H\cup S_K$, where $S_H$ and $S_K$ are generating sets of $H$ and $K$, respectively. We can partition $\cay{H*K}{S_H\cup S_K}$ as follows. Any element $x\in H*K$ belongs to a copy of $\cay{H}{S_H}$ which, when removed, divides the Cayley graph $\cay{H*K}{S_H\cup S_K}$ into $|H|$ connected components that we number $P_0,\ldots, P_{|H|-1}$. To each of these sets $P_i$ we add the unique vertex of the original copy of $\cay{H}{S_K}$ to which it is connected in $\cay{H*K}{S_H\cup S_K}$, and we call them the \textit{petals} associated to this copy. A similar decomposition holds when considering $x$ as forming part of a copy of $\cay{K}{S_K}$, now obtaining $|K|$ petals associated to each copy of this finite subgraph. This decomposition by petals will allow us to compare the solutions to different instances of the TSP in $\cay{H*K}{S_H\cup S_K}$, and hence the word metric of the associated lamplighter group.

\subsection{A lamplighter group with uniformly bounded depth for standard generating sets}\label{subsection: lamplighter with uniformly bounded depth}

Fix the lamps group $\Z/2\Z=\langle a\mid a^2\rangle$, and consider the groups $H=\Z/8\Z=\langle b\mid b^8\rangle $ and $K=\Z/2\Z=\langle c\mid c^2\rangle$. We will study the wreath product $\Z/2\Z\wr (H* K)$ with standard generating set $\std=\{a, b,c\}^{\pm 1}$. In what follows we prove that the dead end depth of $\Z/2\Z\wr (H*K)$ is uniformly bounded with respect to $\std$.

The Cayley graph of $\Z/8\Z*\Z/2\Z$ with respect to the generating set $\{b,c\}$ is formed by octagons (the Cayley graph of $\cay{\Z/8\Z}{b}$) joined together by the generator $c$. This graph contains no odd cycles and hence it is bipartite, so that for any edge, one of its extremes is strictly closer to the identity than the other one. In particular, if we order cyclically the vertices of an octagon as $v_0,\ldots,v_7$ with $v_0$ being the closest to the identity element, then $v_4$ is the furthermost one.

As explained at the beginning of this section, once we have a free product we can consider the petal partition induced by each vertex. In this case, each octagon in $\cay{\Z/8\Z*\Z/2\Z}{ \{b,c\}}$ defines a partition of the Cayley graph into $8$ subsets, which we call petals and denote by $P_0,P_1,\ldots,P_7$, chosen in a cyclic order so that $P_0$ is the component containing the identity and $P_4$ is the furthermost one. This partition is illustrated in Figure \ref{fig: partition by petals Z/8Z*Z/2Z}.

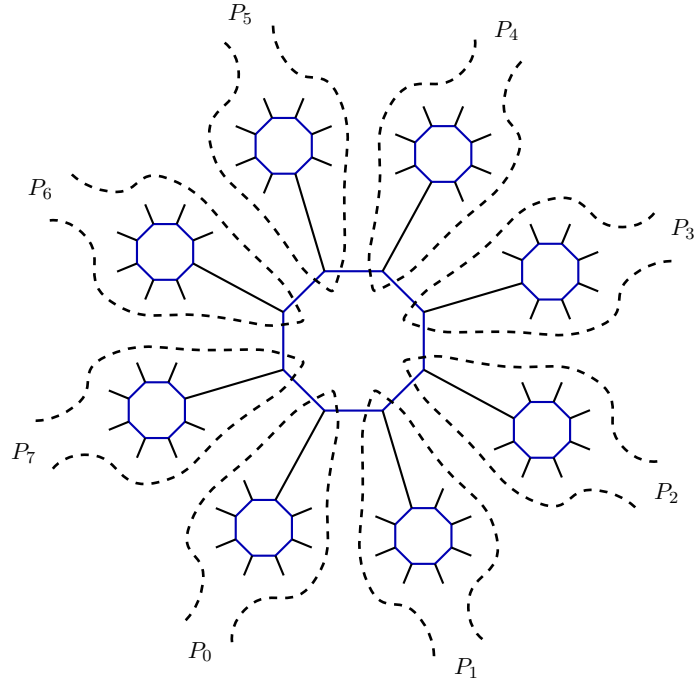
\begin{figure*}[h!]
	\centering
	\begin{tikzpicture}[oct/.style={draw, anchor=west, regular polygon,thick, regular polygon sides=8, outer sep=0,blue!70!black},dot/.style={circle,inner sep=2pt,fill},my star/.append style={star, draw, star points=5,inner sep=0, star point ratio=2},scale=0.8, every node/.style={scale=0.8}]
\node[oct, minimum width=2.5cm] (Ocentral) at (0,0) {};

\foreach \x [count=\i] in {0,45,...,315} {
\node[oct, minimum width=1cm,rotate=\x] (O\i) at ([shift=({\x+22.5:3 cm})]Ocentral.center) {};
}
\draw[thick,-](Ocentral.corner 8) -- (O1.corner 4);
\foreach \x [count=\i,evaluate=\x as \y using { int(\x-1)}] in {2,3,...,8} {
		\draw[thick,-](Ocentral.corner \y) -- (O\x.corner 4);
}
\foreach \x [count=\i] in {1,2,3,...,8} {
	\foreach \a [count=\j,,evaluate=\a as \y using { int(\j+1))}] in {0,45,...,315} {
    \node[] (O\x O\j) at ([shift=({\a+22.5:1 cm})]O\x.center) {};
	}
}
\foreach \a [count=\j] in {2,3,4,5,6,7,8,1}{
	\ifnum \j=4
	\else
	\draw[thick,-](O1.corner \j) -- (O1O\a);
\fi}
\foreach \a [count=\j] in {3,4,5,6,7,8,1,2}{
	\ifnum \j=4
	\else
	\draw[thick,-](O2.corner \j) -- (O2O\a);
	\fi}
\foreach \a [count=\j] in {4,5,6,7,8,1,2,3}{
	\ifnum \j=4
	\else
	\draw[thick,-](O3.corner \j) -- (O3O\a);
	\fi}
\foreach \a [count=\j] in {5,6,7,8,1,2,3,4}{
	\ifnum \j=4
	\else
	\draw[thick,-](O4.corner \j) -- (O4O\a);
	\fi}
\foreach \a [count=\j] in {6,7,8,1,2,3,4,5}{
	\ifnum \j=4
	\else
	\draw[thick,-](O5.corner \j) -- (O5O\a);
	\fi}
\foreach \a [count=\j] in {7,8,1,2,3,4,5,6}{
	\ifnum \j=4
	\else
	\draw[thick,-](O6.corner \j) -- (O6O\a);
	\fi}
\foreach \a [count=\j] in {8,1,2,3,4,5,6,7}{
	\ifnum \j=4
	\else
	\draw[thick,-](O7.corner \j) -- (O7O\a);
	\fi}
\foreach \a [count=\j] in {1,2,3,4,5,6,7,8}{
	\ifnum \j=4
	\else
	\draw[thick,-](O8.corner \j) -- (O8O\a);
	\fi}

\draw [dashed, line width=1pt,bend left=110] plot [smooth, tension=1] coordinates{
	([shift=({5:2 cm})]O1.center) 
	([shift=({-5:1.3 cm})]O1.center) 
	([shift=({-50:1.2 cm})]O1.center)
	([shift=({-130:1.3 cm})]O1.center) 
	([shift=({-156.5:2.3 cm})]O1.center) 
	([shift=({-166.5:2.3 cm})]O1.center)   
	([shift=({-190:1.3 cm})]O1.center)
	([shift=({90:1.2 cm})]O1.center)
	([shift=({35:1.3 cm})]O1.center)
	([shift=({30:2 cm})]O1.center)
} ;
\node at ([shift=({18:2.3 cm})]O1.center) {$P_3$};
\begin{scope}[rotate=45]
\draw [dashed, line width=1pt,bend left=110] plot [smooth, tension=1] coordinates{
	([shift=({5:2 cm})]O2.center) 
	([shift=({-5:1.3 cm})]O2.center) 
	([shift=({-50:1.2 cm})]O2.center)
	([shift=({-130:1.3 cm})]O2.center) 
	([shift=({-156.5:2.3 cm})]O2.center) 
	([shift=({-166.5:2.3 cm})]O2.center)   
	([shift=({-190:1.3 cm})]O2.center)
	([shift=({90:1.2 cm})]O2.center)
	([shift=({35:1.3 cm})]O2.center)
	([shift=({30:2 cm})]O2.center)
} ;
\node at ([shift=({18:2.3 cm})]O2.center) {$P_4$};
\end{scope}
\begin{scope}[rotate=90]
\draw [dashed, line width=1pt,bend left=110] plot [smooth, tension=1] coordinates{
	([shift=({5:2 cm})]O3.center) 
	([shift=({-5:1.3 cm})]O3.center) 
	([shift=({-50:1.2 cm})]O3.center)
	([shift=({-130:1.3 cm})]O3.center) 
	([shift=({-156.5:2.3 cm})]O3.center) 
	([shift=({-166.5:2.3 cm})]O3.center)   
	([shift=({-190:1.3 cm})]O3.center)
	([shift=({90:1.2 cm})]O3.center)
	([shift=({35:1.3 cm})]O3.center)
	([shift=({30:2 cm})]O3.center)
} ;
\node at ([shift=({18:2.3 cm})]O3.center) {$P_5$};
\end{scope}
\begin{scope}[rotate=135]
\draw [dashed, line width=1pt,bend left=110] plot [smooth, tension=1] coordinates{
	([shift=({5:2 cm})]O4.center) 
	([shift=({-5:1.3 cm})]O4.center) 
	([shift=({-50:1.2 cm})]O4.center)
	([shift=({-130:1.3 cm})]O4.center) 
	([shift=({-156.5:2.3 cm})]O4.center) 
	([shift=({-166.5:2.3 cm})]O4.center)   
	([shift=({-190:1.3 cm})]O4.center)
	([shift=({90:1.2 cm})]O4.center)
	([shift=({35:1.3 cm})]O4.center)
	([shift=({30:2 cm})]O4.center)
} ;
\node at ([shift=({18:2.3 cm})]O4.center) {$P_6$};
\end{scope}
\begin{scope}[rotate=180]
\draw [dashed, line width=1pt,bend left=110] plot [smooth, tension=1] coordinates{
	([shift=({5:2 cm})]O5.center) 
	([shift=({-5:1.3 cm})]O5.center) 
	([shift=({-50:1.2 cm})]O5.center)
	([shift=({-130:1.3 cm})]O5.center) 
	([shift=({-156.5:2.3 cm})]O5.center) 
	([shift=({-166.5:2.3 cm})]O5.center)   
	([shift=({-190:1.3 cm})]O5.center)
	([shift=({90:1.2 cm})]O5.center)
	([shift=({35:1.3 cm})]O5.center)
	([shift=({30:2 cm})]O5.center)
} ;
\node at ([shift=({18:2.3 cm})]O5.center) {$P_7$};
\end{scope}
\begin{scope}[rotate=225]
\draw [dashed, line width=1pt,bend left=110] plot [smooth, tension=1] coordinates{
	([shift=({5:2 cm})]O6.center) 
	([shift=({-5:1.3 cm})]O6.center) 
	([shift=({-50:1.2 cm})]O6.center)
	([shift=({-130:1.3 cm})]O6.center) 
	([shift=({-156.5:2.3 cm})]O6.center) 
	([shift=({-166.5:2.3 cm})]O6.center)   
	([shift=({-190:1.3 cm})]O6.center)
	([shift=({90:1.2 cm})]O6.center)
	([shift=({35:1.3 cm})]O6.center)
	([shift=({30:2 cm})]O6.center)
} ;
\node at ([shift=({18:2.3 cm})]O6.center) {$P_0$};
\end{scope}
\begin{scope}[rotate=270]
\draw [dashed, line width=1pt,bend left=110] plot [smooth, tension=1] coordinates{
	([shift=({5:2 cm})]O7.center) 
	([shift=({-5:1.3 cm})]O7.center) 
	([shift=({-50:1.2 cm})]O7.center)
	([shift=({-130:1.3 cm})]O7.center) 
	([shift=({-156.5:2.3 cm})]O7.center) 
	([shift=({-166.5:2.3 cm})]O7.center)   
	([shift=({-190:1.3 cm})]O7.center)
	([shift=({90:1.2 cm})]O7.center)
	([shift=({35:1.3 cm})]O7.center)
	([shift=({30:2 cm})]O7.center)
} ;
\node at ([shift=({18:2.3 cm})]O7.center) {$P_1$};
\end{scope}
\begin{scope}[rotate=315]
\draw [dashed, line width=1pt,bend left=110] plot [smooth, tension=1] coordinates{
	([shift=({5:2 cm})]O8.center) 
	([shift=({-5:1.3 cm})]O8.center) 
	([shift=({-50:1.2 cm})]O8.center)
	([shift=({-130:1.3 cm})]O8.center) 
	([shift=({-156.5:2.3 cm})]O8.center) 
	([shift=({-166.5:2.3 cm})]O8.center)   
	([shift=({-190:1.3 cm})]O8.center)
	([shift=({90:1.2 cm})]O8.center)
	([shift=({35:1.3 cm})]O8.center)
	([shift=({30:2 cm})]O8.center)
} ;
\node at ([shift=({18:2.3 cm})]O8.center) {$P_2$};
\end{scope}

\end{tikzpicture}
	\caption{Each octagon defines a ``partition by petals'' of $\cay{\Z/8\Z*\Z/2\Z}{ \{b,c\}}$.}
	\label{fig: partition by petals Z/8Z*Z/2Z}
\end{figure*}

Now we are ready to prove that the depth of any element of $\Z/2\Z\wr (H*K)$ is uniformly bounded. Indeed, consider an arbitrary element $g\in \Z/2\Z\wr (H* K)$, and write $g=(f,p)$ where $$f\in \bigoplus_{\Z/8\Z*\Z/2\Z}\Z/2\Z \text{ and } p\in \Z/8\Z*\Z/2\Z.$$
As we are trying to find a bound on the depth of $g$, we lose no generality if we suppose that $\|g\|_{\std}>10$; as $\Z/2\Z\wr (\Z/8\Z*\Z/2\Z)$ is an infinite group, the depth of any element is finite and hence the depth of elements on $B_{\std}(e,10)$ is uniformly bounded.

According to our notation, $p$ corresponds to the position of the lamplighter in $\cay{B}{S_B}$ for the element $g$, where $B=\Z/8\Z*\Z/2\Z$ and $S_B=\{b,c\}$. We divide the rest of the proof on two cases.

\begin{itemize}
	
	\item[\textbf{Case 1:}]Suppose there exists $p'\in B_{S_B}(p,10)$ with $f(p')=e_{\Z/2\Z}$ . In other words, there is an unlit lamp at distance at most $10$ from $p$. In such a case we multiply by a word of length at most $21$ which moves the lamplighter position from $p$ to $p'$, lights the corresponding lamp at $p'$, and returns to $p$. This will forcefully result in an element of word length at least $1$ more than $\|g\|_{\std}$. Hence, the depth of $g$ is bounded above by $21$.
	
	\item[\textbf{Case 2:}] $f(p')=a$ for all $p'\in B_{S_B}(p,10)$.	Now we suppose that all lamps within a $10$-neighborhood of $p$ are lit. We look at the octagon to which $p$ belongs, and number its vertices $v_0,v_1,\ldots,v_7$ according to order induced by the cyclic generator $b$. Similarly, we consider the partition by petals $P_0,P_1,\ldots,P_7$ determined by this octagon. In the case we are considering, there are lamps lit at each petal and so
	$$
	\|g\|_{\std}=\left|\supp{f}\right|+\ell_0+\ell_1+\cdots+\ell_7 +7 + \min\left\{ d_B(p,v_1),d_B(p,v_7)\right\},
	$$
	where 
	$$
	\ell_0=\TS{e_B}{v_0}{f|_{P_0}}, \ \ell_i=\TS{v_i}{v_i}{f|_{P_i}} \text{  for  } i=1,\ldots,7,
	$$
	are the lengths of minimal paths traversing the support of $f$ at each petal and finishing in the respective vertices $v_i$. This is is illustrated in Figure \ref{fig: octagons petals TSP}. Note that we have $\min\left\{ d_B(p,v_1),d_B(p,v_7)\right\}\le 3$ so that
	$$\|g\|_{\std}\le\left|\supp{f}\right|+\ell_0+\ell_1+\cdots+\ell_7 +10.
	$$
	
	\begin{figure*}[h!]
		\centering
		\input{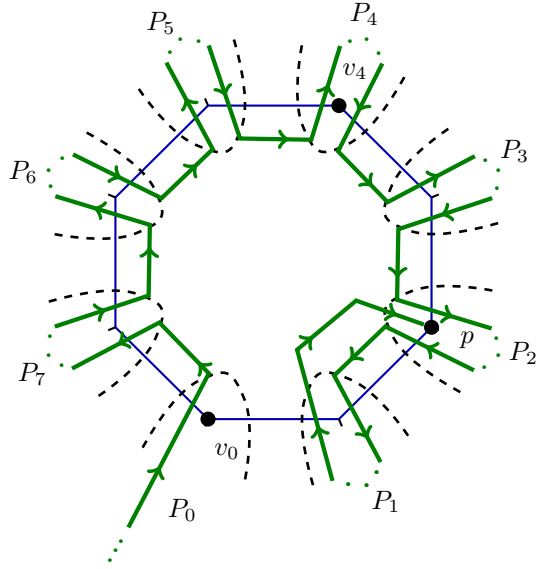}
		\caption{An optimal path for $g$ must visit all petals and return to $p$.}
		\label{fig: octagons petals TSP}
	\end{figure*}
	
	Now we will explain how to increase the word length of $g$ by multiplying by a bounded number of generators. Look in more detail at the octagon $p$ is together with its petal $P_4$. This is defined by a new octagon whose vertices we name $v'_0,v'_1,\ldots,v'_7$ and petals $P'_0,P'_1,\ldots,P'_7$ similarly as we did before. By using at most $9$ generators of $B$, we can move the position of the lamplighter from $p$ to the vertex $v_4$ of the octagon defining $P_4$. Indeed, we need to traverse at most $4$ edges of the octagon of $p$, then the generator $c$ to pass to the next octagon, and then $4$ more edges. We call this new element $g'$. This is depicted in Figure \ref{fig: subpetal P4 in detail}.
	\begin{figure*}[h!]
		\centering
		\input{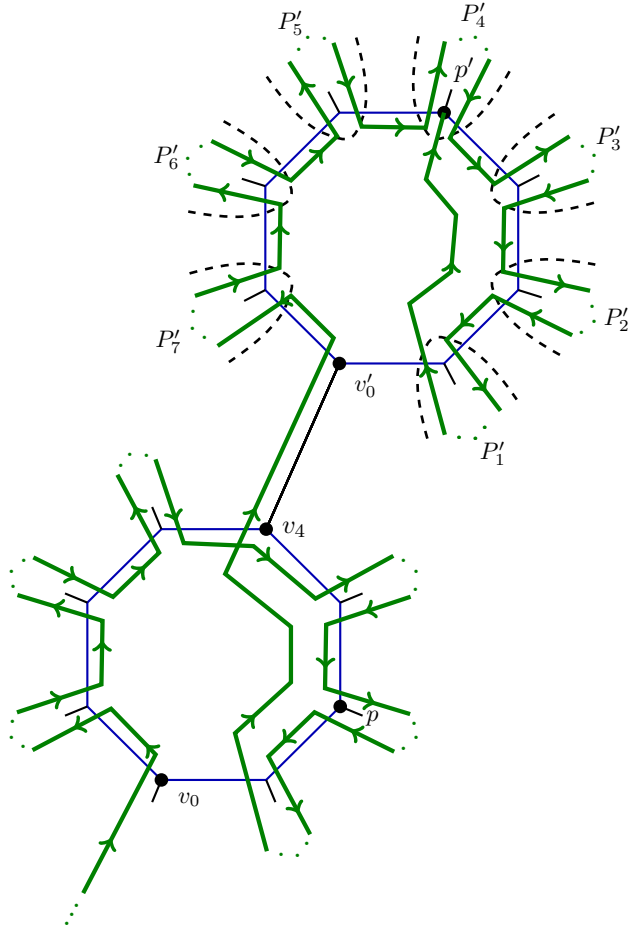}
		\caption{The new position $p'$ of the lamplighter for the element $g'$, with $d_{\std}(g,g')\le 9$.}
		\label{fig: subpetal P4 in detail}
	\end{figure*}
	
	Note that defining analogously the lengths of minimal paths passing through the petals $P'_i$ of the new octagon, 
	$$\ell_i'=\TS{v'_i}{v'_i}{f|_{P'_i}}, \ i=1,2,\ldots,7,
	$$
	we have $\ell_4=\ell'_1+\ell'_2+\cdots+\ell'_7+10$. Indeed, the path of length $\ell_4$ covering the elements of $P_4$ must cross to $v_0'$, pass through each petal $P'_i$, $i=1,\ldots,7$,  while traversing the octagon to finally return to $v_0'$ and afterwards to $v_4$. On the other hand, we see that
	\begin{align*}
	\|g'\|_{\std}=&\left|\supp{f}\right|+\ell_0+\ell_1+\ell_2+\ell_3+\ell_5+\ell_6+\ell_7+10+1+\\
	&+ \ell'_1+\ell'_2+\ell'_3+\ell'_4+\ell'_5+\ell'_6+\ell'_7+10\\
	=&\left|\supp{f}\right|+\ell_0+\ell_1+\cdots+\ell_7+11\\
	\ge&\|g\|_{\std}+1.
	\end{align*}
	This proves that $g$ has depth at most $9$, and hence finishes the proof.
	
\end{itemize}

The existence of generating sets for lamplighter groups with uniformly bounded depth has been proven by Warshall \cite{warshall2008strongly}. However such generating sets are not standard ones, and hence the example explained in this Subsection (and more generally Corollary \ref{cor: free products of finite abelian groups characterization of lamplighter depth} below) seem to provide the first example of lamplighter groups with uniformly bounded depth with respect to standard generators.
\begin{corollary}\label{cor: there exist a lamplighter group with uniformly bounded depth with respect to std gensets}
	Let $(A,S_A)$ be a lamps group of unbounded depth. Then there exist finite groups $(H,S_H)$, $(K,S_K)$ for which $(A\wr (H*K),\std)$ has uniformly bounded depth.
\end{corollary}
\subsection{A general characterization}

Recall that given any two vertices $v,w$ of a finite connected graph $\Gamma$, we denote by $\TS{v}{w}{\Gamma}$ the minimal length of a path in $\Gamma$ which starts at $v$, ends at $w$ and visits all vertices of $\Gamma$.
\begin{definition}\label{def: new constant hamiltonicity finite groups} Let $G$ be a finite group and $S_G$ a finite set. We define
	$$
	\mathscr{H}(G,S_G)=	\max_{\substack{g\in G\backslash\{e_G\}}} \Big\{ \TS{e_G}{g}{G} \Big\} - \TS{e_G}{e_G}{G},
	$$
	where as usual the TSP is considered in the Cayley graph $\cay{G}{S_G}$.
	
\end{definition}

Note that for any generator $g\in S_G$, we have 
$$
\TS{e_G}{e_G}{G}\le \TS{e_G}{g}{G}+1,
$$
so that it always holds that $$\mathscr{H}(G,S_G)\ge -1.$$ Moreover, this lower bound is attained when $\cay{G}{S_G}$ is Hamiltonian-connected. We now state the main result of this section, which characterizes standard generating sets of unbounded depth for lamplighter over a free product of finite groups, in terms of their Hamiltonian differences.

\begin{thm}\label{thm: general dead ends on free products} Let $(H,S_H)$, $(K,S_K)$ be finite groups together with finite symmetric generating sets. Consider the free product $H*K$ with generating set $S_H\cup S_K$, and the lamplighter group $A\wr (H*K)$ with standard generating set $\std$, where $(A,S_A)$ is the lamps group of unbounded depth. Then $A\wr (H*K)$ has uniformly bounded depth with respect to $\std$ if and only if
	
	\begin{equation}\label{eq:general characterization dead ends on free products}
	\mathscr{H}(H,S_H)+\mathscr{H}(K,S_K)\ge 1.
	\end{equation} 
	
\end{thm}
\begin{proof}
	Let us suppose first that Equation \eqref{eq:general characterization dead ends on free products} holds, and let us prove that $A\wr (H*K)$ has uniformly bounded depth. Begin by choosing elements $v\in H\backslash\{e_H\}$, $w\in K\backslash\{e_K\}$ with
	$$
	\TS{e_H}{v}{H}=L_H\coloneqq  \max_{v'\in H\backslash\{e_H\}}\left\{\TS{e_H}{v'}{H}\right\},$$ and $$\TS{e_K}{w}{K}=L_K\coloneqq  \max_{w'\in K\backslash\{e_K\}}\left\{\TS{e_K}{w'}{K}\right\}.
	$$
	Our hypothesis is that 
	\begin{equation}\label{eq: inside the proof general characterization dead ends on free products}
	L_H+L_K\ge \TS{e_H}{e_H}{H}+\TS{e_K}{e_K}{K}+1.
	\end{equation}
	
	Note that Equation \eqref{eq: inside the proof general characterization dead ends on free products} implies that either $$L_H\ge \TS{e_H}{e_H}{H}+1, $$ or $$L_K\ge \TS{e_K}{e_K}{K}+1.$$ Without loss of generality, we suppose this is the case for $L_H$.

	To simplify notation, denote $G\coloneqq H*K$ and $S_G=S_H\cup S_K$. Let $g=(f,p)\in A\wr G$ be any element. If there exists $p'\in B_{S_G}(p,2|K|+2|H|)$ with $f(p')=e_A$, then we can increase the word length of $g$ by moving the lamplighter to $p'$, changing the state of this lamp to a non-trivial element, and then return to $p$. This uses at most $4(|K|+|H|)+1$ generators and since $H$ and $K$ are finite this is a constant.
	
	Now assume that for all $p'\in B_{S_G}(p,2|K|+2|H|)$ we have $f(p')\neq e_A$, so that any path in $\cay{G}{S_G}$ evaluating to $g$ must visit all the elements in said ball at least once.
	
	Consider the copy of $\cay{H}{S_H}$ to which $p$ belongs, and the associated petals $P_0,\ldots,P_{|H|-1}$, where $P_0$ is the petal containing the identity $e_G$. In the same way, we number the vertices of this copy of $H$ as $v_0,\ldots,v_{|H|-1}$ according to their associated petal. Then the word length of $g$ can be expressed as
	$$
	\|g\|_{\std}=\sum_{z\in \supp{f}}\|f(z)\|_{S_A}+\ell_0+\ell_1+\cdots+\ell_{|H|-1}+\TS{e_H}{v(p)}{H},
	$$
	where $\ell_0=\TS{e_K}{v_0}{f|_{P_0}}$, $\ell_i=\TS{v_i}{v_i}{f|_{P_i}}$, $i=1,\ldots,|K|-1$, and $v(p)$ is the vertex of $\cay{H}{S_H}$ that coincides with the position of $p$ inside this copy of $H$. In particular, we have
	$$
	\|g\|_{\std}\le\sum_{z\in \supp{f}}\|f(z)\|_{S_A}+\ell_0+\ell_1+\cdots+\ell_{|H|-1}+ L_H.
	$$
	Now consider the element $g'=gv(p)^{-1}vwv$. That is, we are moving the position of the lamplighter from $p$ to the vertex $v$ (which maximizes the value of $\TS{e_H}{v}{H}$), then to the vertex $w$ (which maximizes the value of $\TS{e_K}{w}{K}$) of the corresponding copy of $\cay{K}{S_K}$, and finally again to the vertex $v$ of the new copy of $\cay{H}{S_H}$. 
	
	As before, the vertices of $\cay{K}{S_K}$ define a partition by petals $P_0',P_1',\ldots,P'_{|K|-1}$ of $\cay{K}{S_K}$, where the identity element $e_K$ belongs to $P_0'$. Again, we number the vertices of this copy of $\cay{K}{S_K}$ by $w'_0,\ldots,w'_{|K|-1}$ according to the petal they define. Defining $$\ell_j'=\TS{w'_j}{w'_j}{f|_{P'_j}}, \ j=1,\ldots,|K|-1,$$
	it follows that if $v=v_i$, for some $i\in \{1,\ldots,|H|-1\}$, then
	$$
	\ell_i=\TS{e_K}{e_K}{K}+\ell'_1+\cdots+\ell'_{|K|-1}.
	$$
	
	Similarly, the new copy of $\cay{H}{S_H}$ gives a new partition of $\cay{K}{S_K}$ into petals $P_0'',P_1'',\ldots,P''_{|H|-1}$ with $e_K\in P_0''$ and the vertices of $H$ numbered as $v_m''$ according to the petal they belong to. Say that $w=w'_r$ for some $r\in \{1,\ldots,|K|-1\}$. Then
	
	$$
	\ell'_r= \TS{e_H}{e_H}{H}+ \ell''_1+\cdots+\ell''_{|H|-1},
	$$
	where we defined $\ell''_h=\TS{v''_h}{v''_h}{f|_{P''_h}}$, $h=1,\ldots, |H|-1$.

	We can express the word length of $g'$ in terms of all these values. Indeed, we have
	
	\begin{align*}
	\|g'\|_{\std}=&\sum_{z\in \supp{f}}\|f(z)\|_{S_A}+ \sum_{j\neq i}\ell_j +\TS{e_H}{v}{H}\ + \\ &+\TS{e_K}{w}{K}+\sum_{h\neq r}\ell'_h+\TS{e_H}{v}{H}+\ell''_1+\cdots+\ell''_{|H|-1}.
	\end{align*}

	Combining the last three equations together with Equation \eqref{eq: inside the proof general characterization dead ends on free products}, we obtain
	\begin{align*}
	\|g'\|_{\std}=&\sum_{z\in \supp{f}}\|f(z)\|_{S_A}+ \sum_{j\neq i}\ell_j +\TS{e_H}{v}{H}\ + \\ &+\TS{e_K}{w}{K}+\sum_{h\neq r}\ell'_h+\TS{e_H}{v}{H}+\ell''_1+\cdots+\ell''_{|H|-1}\\
	=& \sum_{z\in \supp{f}}\|f(z)\|_{S_A}+ \sum_{j\neq i}\ell_j +L_H +L_K+\sum_{h\neq r}\ell'_h+L_H+\ell''_1+\cdots+\ell''_{|H|-1}\\
	\ge & \sum_{z\in \supp{f}}\|f(z)\|_{S_A}+ \sum_{j\neq i}\ell_j +\TS{e_H}{e_H}{H}+\TS{e_K}{e_K}{K}+1+\ \\ &+\sum_{h\neq r}\ell'_h+L_H+\ell''_1+\cdots+\ell''_{|H|-1}\\
	= & \sum_{z\in \supp{f}}\|f(z)\|_{S_A}+ \sum_{j\neq i}\ell_j +\TS{e_K}{e_K}{K}+1+\sum_{h\neq r}\ell'_h+L_H+\ell'_r\\
	= & \sum_{z\in \supp{f}}\|f(z)\|_{S_A}+ \sum_{j\neq i}\ell_j +\TS{e_K}{e_K}{K}+1+\sum_{h}\ell'_h+L_H\\
	= & \sum_{z\in \supp{f}}\|f(z)\|_{S_A}+ \sum_{j\neq i}\ell_j+\ell_i+1+L_H\\
	= & \sum_{z\in \supp{f}}\|f(z)\|_{S_A}+ \sum_{j}\ell_j+L_H+1\\
	\ge & \|g\|_{\std}+1.
	\end{align*}
	
	In the process of obtaining $g'$ we multiplied by at most $2|H|+|K|$ generators, and so we conclude that the depth of any element is bounded by a uniform constant. This proves the first direction of the proposition.

	Now let us prove that if Equation \eqref{eq:general characterization dead ends on free products} does not hold then the associated lamplighter group has unbounded depth. From the hypothesis we deduce the inequalities
	$$
	\TS{e_H}{e_H}{H}-1\le \max_{v\in H\backslash\{e_H\}}\left\{\TS{e_H}{v}{H}\right\}\le \TS{e_H}{e_H}{H}+1,
	$$
	and 
	$$
	\TS{e_K}{e_K}{K}-1\le \max_{w\in K\backslash\{e_K\}}\left\{\TS{e_K}{w}{K}\right\}\le \TS{e_K}{e_K}{K}+1.
	$$
	
	With this, the possible values for these two maximums are illustrated in Table \ref{table: possible values TS free products}. Here we denote $L_H= \max_{v\in H\backslash\{e_H\}}\left\{\TS{e_H}{v}{H}\right\}$ and 
	
	$L_K= \max_{w\in K\backslash\{e_K\}}\left\{\TS{e_K}{w}{K}\right\}$.
	\begin{table}[h!]
		\centering
		\begin{tabular}{|l|l|l|l|}
			\hline
			\diagbox[width=3.5cm]{$L_K$}{$L_H$}& $\TS{e_H}{e_H}{H}-1$ & $\TS{e_H}{e_H}{H}$ & $\TS{e_H}{e_H}{H}+1$ \\ \hline
			$\TS{e_K}{e_K}{K}-1$ &           Case $2$           &     Case $2$               &           Case $1$           \\ \hline
			$\TS{e_K}{e_K}{K}$   &           Case $2$           &      Case $2$              &            Impossible          \\ \hline
			$\TS{e_K}{e_K}{K}+1$ &          Case $1$            &          Impossible          &    Impossible                  \\ \hline
		\end{tabular}
		\caption{Possible values for the solutions of the TSP inside each finite graph.}
		\label{table: possible values TS free products}
	\end{table}
	Each possible combination of values of $L_H$ and $L_K$ will be covered in two separate cases, as shown in the table. Recall that for convenience, we defined $G=H*K$ and $S_G=S_H\cup S_K.$
	\begin{itemize}
		\item[\textbf{Case 1.}] Suppose that $L_H=\TS{e_H}{e_H}{H}+1$, so that we must have $$L_K= \TS{e_K}{e_K}{K}-1.$$ Choose $v\in H$ such that $\TS{e_H}{v}{H}=L_H$.
		
		Consider $a\in (A,S_A)$ of depth at least $n$, and as usual define the element $g=(f,v)$, where $f(x)=a$ if $\|x\|_K\le n$ and $f(x)=1_A$ otherwise. We will prove that $g$ has depth at least $n-1$. Similar to how we approached the case of free groups, changing the lamp states cannot increase word length so the proof will follow from the following claim.

		We claim that for any $x\in G$ with $1\le \|x\|_{S_G}\le n-1$, we have $\|gx\|_{\std}\le \|g\|_{\std}-1$ if $x$ finishes with an element of $K\backslash\{e_K\}$ and $\|gx\|_{\std}\le \|g\|_{\std}$ if $x$ finishes with an element of $H\backslash\{e_H\}$.

		Indeed, let us do an inductive proof.
		If $x\in H$, then 
		\begin{equation*}
		\|gx\|_{\std}=\|g\|_{\std}+\TS{e_H}{x}{H}-\TS{e_H}{v}{H}\le \|g\|_{\std}.
		\end{equation*}
		Similarly, if $x\in K$ then 
		\begin{align*}
		\|gx\|_{\std}&=\|g\|_{\std}+\TS{e_K}{x}{K}-\TS{e_K}{e_K}{K}\\&\le \|g\|_{\std}+\TS{e_H}{e_H}{H}-L_H\\ &= \|g\|_{\std}-1\\&\le \|g\|_{\std}.
		\end{align*}

		Now suppose that $x$ is of the form $x'hk$, for $h\in H\backslash\{e_H\}$, $k\in \backslash\{e_K\}$. Then looking at the petal decomposition of $\cay{H}{S_H}$ and using the inductive hypothesis,
		\begin{align*}
		\|gx\|_{\std}&=\|gx'hk\|_{\std}\\
		&=\|gx'h\|_{\std}+\TS{e_K}{k}{K}-\TS{e_K}{e_K}{K}\\
		&\le \|g\|_{\std}-1. 
		\end{align*}
		
		On the other hand, if $x$ is of the form $x'kh$ for $h\in H\backslash\{e_H\}$, $k\in \backslash\{e_K\}$. Similarly to the above we have,
		\begin{align*}
		\|gx\|_{\std}&=\|gx'kh\|_{\std}\\
		&=\|gx'k\|_{\std}+\TS{e_H}{h}{H}-\TS{e_H}{e_H}{H}\\
		&\le \|gx'k\|_{\std}+1\\
		&\le \|g\|_{\std}.
		\end{align*}
		This finishes the proof of the first case. By symmetry of $H$ and $K$, the proof for the case where $L_K=\TS{e_K}{e_K}{K}+1$ is completely analogous.
		\item[\textbf{Case 2.}] 
		Now suppose that $L_H\le \TS{e_H}{e_H}{H}$ and $L_K\le \TS{e_K}{e_K}{K}$. A similar inductive argument to the one given in the first case proves that if we define $g$ as in Case $1$, now for any $x\in K$ with $1\le \|x\|_{S_K}\le n$ we have $\|gx\|_{\std}\le \|g\|_{\std}$.
	\end{itemize}

	With this, we see that the element $g$ constructed has depth at least $n-1$. As $n$ was arbitrary, we conclude that $A\wr (H*K)$ has unbounded depth with respect to $\std$.
\end{proof}

If $\cay{G}{S_G}$ is a cycle, then the value of $\TS{e_G}{e_G}{G}$ is always equal to $|G|$, which is attained with a path starting at $e_G$ and traversing the cycle. On the other hand, the value of $\max_{\substack{g\in G\backslash\{e_G\}}}\Big\{ \TS{e_G}{g}{G}\Big\}$ is $$|G|-1+\left\lfloor \frac{|G|}{2}\right\rfloor-1=|G|+\left\lfloor \frac{|G|}{2}\right\rfloor-2.$$ This is the length of a path starting at $e_G$, doing the cycle up to the last vertex before returning to $e_G$, and then going back to an element $g\in G$ at distance $\left\lfloor \frac{|G|}{2}\right\rfloor$ from the identity.

The above implies that $\displaystyle \mathscr{H}(G,S_G)=\left\lfloor \frac{|G|}{2}\right\rfloor-2$, and so in particular we have that $
\mathscr{H}(\Z/8\Z,\{b\})+\mathscr{H}(\Z/2\Z,\{b\})=1$, so that the example of Subsection \ref{subsection: lamplighter with uniformly bounded depth} is consistent with Theorem \ref{thm: general dead ends on free products}.

More generally, for any pair of cyclic groups $(H,S_H)$ and $(K,S_K)$, with cyclic generating sets, Condition \eqref{eq:general characterization dead ends on free products} holds if and only if
$$
\left\lfloor\frac{|H|}{2} \right\rfloor+\left\lfloor\frac{|K|}{2} \right\rfloor\ge 5.
$$
Denoting the orders of $H$ and $K$ by $o_H$ and $o_K$, respectively, we have that Theorem \ref{thm: general dead ends on free products} implies the lamplighter over $H*K$ has unbounded depth with respect to the standard generating set if and only if
\begin{enumerate}
	\item $(o_H,o_K)\in \{(6,4),(6,5),(6,6),(7,4),(7,5),(7,6),(7,7)\}$, or
	\item  $o_H\ge 8$ and $o_K\ge 2$.
\end{enumerate} 

\begin{example}\label{example: free product unbounded depth arbitrarily far away position} Consider $H=\Z/4\Z=\langle b\rangle$ and $K=\Z/4\Z=\langle c \rangle$, so that according to Corollary \ref{cor: free products of finite abelian groups characterization of lamplighter depth} the lamplighter group $A\wr (H*K)$ has unbounded depth. By following the proof of Theorem \ref{thm: general dead ends on free products}, it is possible to see that for any element $x\in \langle b^2,c^2\rangle\leqslant H*K$, there is a dead end $g\in A\wr (H*K)$ of arbitrarily large depth of the form $g=(f,x)$. That is, the position $x$ of the lamplighter in a dead end can be at an arbitrary distance from the identity of $H*K$. This is a difference with the behavior of dead ends of lamplighters over free groups, where the position of the lamplighter for a dead end is necessarily the identity element (Proposition \ref{prop: characterization of dead ends on lamplighters over free groups}).
	\begin{figure}[h!]
		\centering
		\begin{tikzpicture}[cuad_blue/.style={draw, anchor=west, regular polygon,very thick, regular polygon sides=4, outer sep=0,orange!70!black},cuad_green/.style={draw, anchor=west, regular polygon, very thick, regular polygon sides=4, outer sep=0,green!70!black},dot/.style={circle,inner sep=2.5pt,fill},my star/.append style={star, draw, star points=5,inner sep=0, star point ratio=2},scale=0.6, every node/.style={scale=0.6}]

\node[cuad_blue, minimum width=4.5cm] (Ocentralblue) at (0,0) {};
\node[cuad_green, minimum width=4.5cm] (Ocentralgreen) at ([shift=({213.5:5.7 cm})]Ocentralblue.center) {};
\node[dot, label={[above left]\Huge$e$}] at (Ocentralblue.corner 3) {};
\node[cuad_green, minimum width=2cm] (Ogreen11) at ([shift=({142.5:3.78 cm})]Ocentralblue.center) {};
\node[cuad_green, minimum width=2cm] (Ogreen12) at ([shift=({55.5:2.8 cm})]Ocentralblue.center) {};
\node[cuad_green, minimum width=2cm] (Ogreen13) at ([shift=({-55.5:2.8 cm})]Ocentralblue.center) {};
\node[cuad_blue, minimum width=2cm] (Oblue11) at ([shift=({142.5:3.78 cm})]Ocentralgreen.center) {};
\node[cuad_blue, minimum width=2cm] (Oblue12) at ([shift=({-142.5:3.78 cm})]Ocentralgreen.center) {};
\node[cuad_blue, minimum width=2cm] (Oblue13) at ([shift=({-55.5:2.8 cm})]Ocentralgreen.center) {};

\node[cuad_blue, minimum width=1cm] (Oblue111) at ([shift=({143:1.78 cm})]Ogreen11.center) {};
\node[cuad_blue, minimum width=1cm] (Oblue112) at ([shift=({-143:1.78 cm})]Ogreen11.center) {};
\node[cuad_blue, minimum width=1cm] (Oblue112) at ([shift=({56.5:1.28 cm})]Ogreen11.center) {};

\node[cuad_blue, minimum width=1cm] (Oblue121) at ([shift=({143:1.78 cm})]Ogreen12.center) {};
\node[cuad_blue, minimum width=1cm] (Oblue122) at ([shift=({-56:1.28 cm})]Ogreen12.center) {};
\node[cuad_blue, minimum width=1cm] (Oblue122) at ([shift=({56.5:1.28 cm})]Ogreen12.center) {};

\node[cuad_blue, minimum width=1cm] (Oblue131) at ([shift=({-56:1.28 cm})]Ogreen13.center) {};
\node[cuad_blue, minimum width=1cm] (Oblue132) at ([shift=({-143:1.78 cm})]Ogreen13.center) {};
\node[cuad_blue, minimum width=1cm] (Oblue132) at ([shift=({56.5:1.28 cm})]Ogreen13.center) {};

\node[cuad_green, minimum width=1cm] (Ogreen111) at ([shift=({143:1.78 cm})]Oblue11.center) {};
\node[cuad_green, minimum width=1cm] (Ogreen112) at ([shift=({-143:1.78 cm})]Oblue11.center) {};
\node[cuad_green, minimum width=1cm] (Ogreen112) at ([shift=({56.5:1.28 cm})]Oblue11.center) {};

\node[cuad_green, minimum width=1cm] (Ogreen121) at ([shift=({143:1.78 cm})]Oblue12.center) {};
\node[cuad_green, minimum width=1cm] (Ogreen122) at ([shift=({-56:1.28 cm})]Oblue12.center) {};
\node[cuad_green, minimum width=1cm] (Ogreen122) at ([shift=({-143:1.78 cm})]Oblue12.center) {};

\node[cuad_green, minimum width=1cm] (Ogreen131) at ([shift=({-56:1.28 cm})]Oblue13.center) {};
\node[cuad_green, minimum width=1cm] (Ogreen132) at ([shift=({-143:1.78 cm})]Oblue13.center) {};
\node[cuad_green, minimum width=1cm] (Ogreen132) at ([shift=({56.5:1.28 cm})]Oblue13.center) {};

\node[dot] at (Ocentralblue.corner 1) {};
\node[dot] at (Ogreen12.corner 1) {};
\node[dot] at (Oblue122.corner 1) {};
\node[dot] at (Oblue12.corner 1) {};
\node[dot] at (Ogreen122.corner 1) {};
\node[dot] at (Ogreen122.corner 3) {};
\node[font=\Huge] at ([shift=({45:0.8 cm})]Oblue122.corner 1) {$\cdot$};
\node[font=\Huge] at ([shift=({45:1 cm})]Oblue122.corner 1) {$\cdot$};
\node[font=\Huge] at ([shift=({45:1.2 cm})]Oblue122.corner 1) {$\cdot$};
\node[font=\Huge] at ([shift=({-140:0.8 cm})]Ogreen122.corner 3) {$\cdot$};
\node[font=\Huge] at ([shift=({-140:1 cm})]Ogreen122.corner 3) {$\cdot$};
\node[font=\Huge] at ([shift=({-140:1.2 cm})]Ogreen122.corner 3) {$\cdot$};
\end{tikzpicture}
		\caption{The lamplighter position of a dead end of arbitrary depth of \\ $(A\wr (\Z/4\Z*\Z/4\Z),\std)$ can be any element inside the subgroup $\langle b^2,c^2\rangle$.}
		\label{fig: Z/4Z*Z/4Z possible positions for lamplighter in a dead end}
	\end{figure}
\end{example}

Recall that a finite graph is said to be Hamiltonian-connected if any pair of distinct vertices can be joined by a Hamiltonian path. This is not possible in a bipartite graph, so in that case the strongest possible condition is being Hamiltonian-laceable: having Hamiltonian paths between any two vertices of distinct partite sets of the graph.

\begin{lemma}\label{lem: tsp in hamiltonian connected and hamiltonian laceable graphs}
	Let $(G,S_G)$ be a finite group together with a finite generating set. If $\cay{G}{S_G}$ is either Hamiltonian-connected or Hamiltonian-laceable, then
	$$
	\mathscr{H}(G,S_G)\le 0.
	$$
\end{lemma}
\begin{proof}
	If $\cay{G}{S_G}$ is Hamiltonian-connected, then $\TS{e_G}{e_G}{G}=|G|+1$, while for any $g\in G\backslash\{e_G\}$ we have $\TS{1_G}{g}{G}=|G|$. This implies that $\mathscr{H}(G,S_G)=-1$.
	
	Now suppose $\cay{G}{S_G}$ is Hamiltonian-laceable, so that in particular it is bipartite. Write $G=A\cup B$ where $A$ and $B$ form a partition with $e_G\in G$. By considering a Hamiltonian path from $e_G\in A$ to a generator in $S_G\subseteq B$, we see that $\TS{e_G}{e_G}{G}=|G|+1$.
	
	Now for any $g\in G\backslash\{e_G\}$, we consider two cases. If $g\in B$, then there is a Hamiltonian path from $e_G$ to $g$ so that $\TS{e_G}{g}{G}=|G|$. On the contrary, if $g\in A$ then there is a Hamiltonian path from $e_G$ to a neighbor of $g$ and hence $\TS{e_G}{g}{G}=|G|+1$. We conclude that $\mathscr{H}(G,S_G)=0$.
\end{proof}

\begin{corollary}\label{cor: free product of two hamiltonian connected or hamiltonian laceable has unbounded depth for lamplighter} Suppose that $(H,S_H)$ and $(K,S_K)$ are two finite groups with finite generating sets, which are both either Hamiltonian-connected or Hamiltonian-laceable. Then for any group $(A,S_A)$ with unbounded depth, the lamplighter group $A\wr(H*K)$ with the corresponding standard generating set $\std$ has unbounded depth. 
\end{corollary}
\begin{proof}
	Lemma \ref{lem: tsp in hamiltonian connected and hamiltonian laceable graphs} implies that
	$$
	\mathscr{H}(K,S_K)+\mathscr{H}(H,S_H)\le 0,
	$$
	which is precisely the negation of Condition \eqref{eq:general characterization dead ends on free products} in Theorem \ref{thm: general dead ends on free products}.
\end{proof}
The following corollary characterizes the case of free products of finite abelian groups, generalizing our previous comments about cyclic groups, and covering all possible Cayley graphs. One can think of this corollary as saying that among lamplighters over the free products of two finite abelian groups, the only way to get uniformly bounded depth is for the finite groups forming the base to be sufficiently long cycles.

\begin{corollary}\label{cor: free products of finite abelian groups characterization of lamplighter depth} Suppose that $(H,S_H)$ and $(K,S_K)$ are two finite abelian groups. For any group $(A,S_A)$ with unbounded depth, consider the lamplighter group $A\wr(H*K)$ with the corresponding standard generating set $\std$.  We now list all possible cases for $H$ and $K$.
	\begin{enumerate}
		\item If $|H|=1$ or $|K|=1$, then $(A\wr(H*K),\std)$ has unbounded depth.
		
		\item If $|H|\in \{2,3\}$ (resp. $|K|\in \{2,3\}$), then
		\begin{enumerate}
			\item if $\cay{K}{S_K}$ (resp. $\cay{H}{S_H}$) is a cycle of length at least $8$, then $(A\wr(H*K),\std)$ has uniformly bounded depth, and
			
			\item otherwise $(A\wr(H*K),\std)$ has unbounded depth.
		\end{enumerate}
	\end{enumerate}
	Now suppose that $|H|,|K|\ge 4$.
	\begin{enumerate}\setcounter{enumi}{2}
		\item If neither $\cay{H}{S_H}$ nor $\cay{K}{S_K}$ are cycles, then $(A\wr(H*K),\std)$ has unbounded depth.
		\item Suppose that $\cay{H}{S_H}$ is a cycle.
		\begin{enumerate}
			\item If $|H|\in \{4,5\}$, then $(A\wr(H*K),\std)$ has uniformly bounded depth if and only if $\cay{K}{S_K}$ is a cycle of length at least $6$.
			
			\item If $|H|\in \{6,7\}$, then $(A\wr(H*K),\std)$ has uniformly bounded depth if and only if $\cay{K}{S_K}$ is a cycle or bipartite.
			
			\item If $|H|\ge 8$, then $(A\wr(H*K),\std)$ has uniformly bounded depth.
		\end{enumerate}
		\item An analogous statement to (4) holds when $\cay{K}{S_K}$ is a cycle.
	\end{enumerate}
\end{corollary}
\begin{proof}
	\begin{enumerate}
		\item If $|H|=1$ or $|K|=1$, then $H*K$ is a finite group, and the result follows from Proposition \ref{prop: lamplighter over finite group deep dead ends if and only if lamps have deep dead ends}.
		
		\item Suppose that $|H|\in \{2,3\}$. Then $H\cong \Z/2\Z$ or $H\cong\Z/3\Z$, and in both cases it holds that $\mathscr{H}(H,S_H)=-1$.
		\begin{enumerate}
			\item If $\cay{K}{S_K}$ is a cycle of length $\ell\ge 8$, number its vertices cyclically $w_0,\ldots,w_{\ell-1}$ where $w_0$ is the identity element. Then the vertex indexed by $\lfloor \frac{\ell}{2}\rfloor$ satisfies $$\TS{e_K}{w_{\lfloor\frac{\ell}{2}\rfloor} }{K}\ge\TS{e_K}{e_K}{K}+2,$$ and so
			$\mathscr{H}(K,S_K)\ge 2$. This implies that
			$$
			\mathscr{H}(H,S_H)+\mathscr{H}(K,S_K)\ge -1+2=1,
			$$
			and hence Condition \eqref{eq:general characterization dead ends on free products} holds.
			
			\item In any other case, Proposition \ref{prop: cayley graphs of abelian finite group are either cycles, Hamiltonian connected or Hamiltonian laceable } implies that $\cay{K}{S_K}$ is either Hamiltonian-connected or Hamiltonian-laceable. In both cases, Corollary \ref{cor: free product of two hamiltonian connected or hamiltonian laceable has unbounded depth for lamplighter} proves that the corresponding lamplighter group has unbounded depth.
			
		\end{enumerate}
		
		\item If neither $\cay{H}{S_H}$ nor $\cay{K}{S_K}$, then both of these graphs are either Hamiltonian-connected or Hamiltonian laceable thanks to Proposition \ref{prop: cayley graphs of abelian finite group are either cycles, Hamiltonian connected or Hamiltonian laceable }. Then Corollary \ref{cor: free product of two hamiltonian connected or hamiltonian laceable has unbounded depth for lamplighter} implies that $(A\wr(H*K),\std)$ has unbounded depth.
		\item Now we suppose that $\cay{H}{S_H}$ is a cycle of length at least $4$.
		\begin{enumerate}
			\item If $|H|\in \{4,5\}$, then $\max_{v\in H\backslash\{e_H\}}\Big\{\TS{e_H}{v}{H} \Big \}=\TS{e_H}{e_H}{H}$ and hence $\mathscr{H}(H,S_H)=0$. Then if $\cay{K}{S_K}$ is a cycle of length at least $6$ we have $\max_{w\in K\backslash\{e_K\}}\Big\{\TS{e_K}{w}{K} \Big\}= \TS{e_K}{e_K}{K}+1$, and in any other case Proposition \ref{prop: cayley graphs of abelian finite group are either cycles, Hamiltonian connected or Hamiltonian laceable } together with Lemma \ref{lem: tsp in hamiltonian connected and hamiltonian laceable graphs} show that $$\max_{w\in K\backslash\{e_K\}}\Big\{\TS{e_K}{w}{K} \Big\}\le \TS{e_K}{e_K}{K}.$$ In the first case Condition \eqref{eq:general characterization dead ends on free products} in Theorem \ref{thm: general dead ends on free products} is satisfied, while on the second its negation holds.
			
			\item If $|H|\in \{6,7\}$, then $\max_{v\in H\backslash\{e_H\}}\Big\{\TS{e_H}{v}{H} \Big \}=\TS{e_H}{e_H}{H}+1$. If $\cay{K}{S_K}$ is a cycle, it must have length at least $4$ and so
			$$
			\max_{w\in K\backslash\{e_K\}}\Big\{\TS{e_K}{w}{K} \Big\}\ge \TS{e_K}{e_K}{K}.
			$$
			On the other hand, if $\cay{K}{S_K}$ is not a cycle then
			$$
			\max_{w\in K\backslash\{e_K\}}\Big\{\TS{e_K}{w}{K} \Big\}= \TS{e_K}{e_K}{K},
			$$ if $\cay{K}{S_K}$ is bipartite, and 
			$$
			\max_{w\in K\backslash\{e_K\}}\Big\{\TS{e_K}{w}{K} \Big\}= \TS{e_K}{e_K}{K}-1,
			$$
			otherwise. The first two cases satisfy Condition \eqref{eq:general characterization dead ends on free products} in Theorem \ref{thm: general dead ends on free products} while the third one does not.
			
			\item If $|H|\ge 8$, then $\max_{v\in H\backslash\{e_H\}}\Big\{\TS{e_H}{v}{H} \Big \}\ge\TS{e_H}{e_H}{H}+2$. In general, we have that 
			$$
			\max_{w\in K\backslash\{e_K\}}\Big\{\TS{e_K}{w}{K} \Big\}\ge \TS{e_K}{e_K}{K}-1,
			$$
			so that Condition \eqref{eq:general characterization dead ends on free products} in Theorem \ref{thm: general dead ends on free products} is always satisfied.
		\end{enumerate}
		\item An analogous proof replacing $H$ by $K$ and vice-versa proves the analogous statement to the above.\end{enumerate}
	\phantom{analogous statement to the above.}
\end{proof}

\section*{Acknowledgments}
I would like to thank my advisor Anna Erschler for very helpful discussions, encouragement, and corrections. I would also like to thank Romain Tessera for some suggestions on the redaction of this paper. I thank the anonymous referee for their comments and corrections. This project has received funding from the European Union’s Horizon 2020 research and innovation programme under the Marie Sk\l{}odowska-Curie grant agreement N\textsuperscript{\underline{o}} 945322.

\bibliographystyle{ws-ijac}
\bibliography{biblio}
\end{document}